\documentclass[12pt]{amsart}

\usepackage{color}
\usepackage{amsmath}
\usepackage{amsthm}
\usepackage{amssymb}
\usepackage{graphicx}
\usepackage{enumerate}
\usepackage{amsfonts}
\usepackage{mathrsfs}
\usepackage{parskip}
\usepackage{mathdots}
\usepackage{color}

\usepackage{times,a4wide}

\numberwithin{equation}{section}
\theoremstyle{plain}
\newtheorem{Proposition}[equation]{Proposition}
\newtheorem{Corollary}[equation]{Corollary}
\newtheorem*{Corollary*}{Corollary}
\newtheorem{Theorem}[equation]{Theorem}
\newtheorem*{Theorem*}{Theorem}

\theoremstyle{definition}
\newtheorem{Definition}[equation]{Definition}

\def\HH{\mathscr{H}}
\def\MM{\mathscr{M}}
\def\FF{\mathscr{F}}
\def\DD{\mathscr{D}}
\def\C{\mathbb{C}}
\def\R{\mathbb{R}}
\def\D{\mathbb{D}}
\def\T{\mathbb{T}}
\def\N{\mathbb{N}}
\def\Z{\mathbb{Z}}

\def\K{\mathcal{K}}
\def\phi{\varphi}
\newcommand{\vp}{\varphi}
\newcommand{\clos}{\operatorname{clos}}

\newcommand{\dist}{\operatorname{dist}}
\renewcommand{\ker}{\operatorname{Ker}}

\renewcommand{\Re}{\operatorname{Re}}
\renewcommand{\Im}{\operatorname{Im}}

\newcommand{\sinc}{\operatorname{sinc}}
\newcommand{\beqa}{\begin{eqnarray*}}
\newcommand{\eeqa}{\end{eqnarray*}}
\newcommand{\lra}{\longrightarrow}

\title{A survey on reverse Carleson measures}

\author[Fricain]{Emmanuel Fricain}
 \address{Laboratoire Paul Painlev\'e, Universit\'e Lille 1, 59 655 Villeneuve d'Ascq C\'edex }
 \email{emmanuel.fricain@math.univ-lille1.fr}

\author[Hartmann]{Andreas Hartmann}
\address{Institut de Math\'ematiques de Bordeaux, Universit\'e Bordeaux 1, 351 cours de la Lib\'eration 33405 Talence C\'edex, France}
\email{Andreas.Hartmann@math.u-bordeaux1.fr}

\author[Ross]{William T. Ross}
	\address{Department of Mathematics and Computer Science, University of Richmond, Richmond, VA 23173, USA}
	\email{wross@richmond.edu}



\begin{document}

\begin{abstract}
This is a survey on reverse Carleson measures for various Hilbert spaces of analytic functions. These spaces include the Hardy, Bergman, certain harmonically weighted Dirichlet, Paley-Wiener, Fock, model (backward shift invariant), and de Branges-Rovnyak spaces. The reverse Carleson measure for backward shift invariant subspaces in the non-Hilbert situation is new.

\end{abstract}

\maketitle

\section{Introduction}
Suppose that $\mathscr{H}$ is a Hilbert space of analytic functions on the open unit disk $\D = \{z \in \C: |z| < 1\}$ endowed with a norm $\|\cdot\|_{\mathscr{H}}$. 
If $\mu \in M_{+}(\D^{-})$, the positive finite Borel measures on the closed unit disk $\D^{-} = \{z \in \C: |z| \leqslant 1\}$, we say that $\mu$ is a {\em Carleson measure} for $\mathscr{H}$ when
\begin{equation}\label{carldirect}
\|f\|_{\mu} \lesssim \|f\|_{\mathscr{H}} \quad \forall f \in \mathscr{H},
\end{equation} 
and a {\em reverse Carleson measure} for $\mathscr{H}$ when 
\begin{equation}\label{carlreverse}
 \|f\|_{\mathscr{H}} \lesssim \|f\|_{\mu} \quad \forall f\in \mathscr{H}.
\end{equation}
Here we use the notation
$$\|f\|_{\mu} := \left(\int_{\D^{-}} |f|^2 d\mu\right)^{\tfrac{1}{2}}$$ for the $L^2(\mu)$ norm of $f$ and the notation $\|f\|_{\mu} \lesssim \|f\|_{\mathscr{H}}$ to mean there is a constant $c_{\mu} > 0$ such that $\|f\|_{\mu} \leqslant c_{\mu} \|f\|_{\mathscr{H}}$ for every $f \in \mathscr{H}$ (similarly for the inequality  $\|f\|_{\mathscr{H}} \lesssim \|f\|_{\mu}$). We will use the notation 
$\|f\|_{\mu} \asymp \|f\|_{\mathscr{H}}$ when $\mu$ is both a Carleson and a reverse Carleson measure. There is of course the issue of how we define $f$ $\mu$-a.e. on $\T=\partial\D$ so that $\|f\|_\mu$ makes sense; but this will be discussed later. 

Carleson measures for many Hilbert (and Banach) spaces of analytic functions have been well studied for many years now. Due to the large literature on this subject,  it is probably impossible to give a complete account of these results. Carleson measures make, and continue to make, important connections to many areas of analysis such as operator theory, 
interpolation, boundary behavior problems, and Bernstein inequalities and they have certainly proved their worth. We will mention a few of these results as they relate to the lesser known topic, and the focus of this survey, of reverse Carleson measures.

Generally speaking, Carleson measures $\mu$ are often characterized by the amount of mass that $\mu$ places
on a {\em Carleson window} 
$$S_I :=\big\{z \in \D^{-}:1-|I|\leqslant |z|\leqslant 1, \frac{z}{|z|} \in I\big\}$$ 
relative to the length $|I|$ of the side $I$ of that window,
i.e., whether or not there exists positive constants $C$ and $\alpha$ such that 
\begin{eqnarray}\label{carlcond}
 \mu(S_I)\leqslant C |I|^{\alpha}.
\end{eqnarray}
for all
arcs $I \subset \T = \partial \D$. We will write this as $\mu(S_I) \lesssim |I|^{\alpha}$.

When $\mathscr{H}$ is a reproducing kernel Hilbert space, it is often the case that the Carleson condition in \eqref{carldirect} can be equivalently rephrased in terms of the, seemingly weaker, testing condition 
\begin{equation}\label{RKThesis-Intro}
\|k_{\lambda}^{\mathscr{H}}\|_{\mu} \lesssim \|k_{\lambda}^{\mathscr{H}}\|_{\mathscr{H}} \quad \forall \lambda \in \D,
\end{equation}
where $k_{\lambda}^{\mathscr{H}}$ is the reproducing kernel function for $\mathscr{H}$. This testing condition (where \eqref{RKThesis-Intro} implies \eqref{carldirect})  is often called the {\em reproducing kernel thesis} (RKT). 

It is natural to ask as to whether or not reverse Carleson measures on $\mathscr{H}$ can be characterized by replacing the conditions in \eqref{carlcond} and \eqref{RKThesis-Intro} with the analogous ``reverse'' conditions 
$$ \mu(S_I)\gtrsim |I|^{\alpha} \quad \mbox{or} \quad \|k_{\lambda}^{\mathscr{H}}\|_{\mu} \gtrsim \|k_{\lambda}^{\mathscr{H}}\|_{\mathscr{H}}.$$ We will explore when this happens. 

Reverse Carleson measures probably first appeared under the broad heading of  ``sampling measures'' for $\mathscr{H}$, in other words, measures $\mu$ for which 
$$\|f\|_{\mathscr{H}} \asymp \|f\|_{\mu} \quad \forall f \in \mathscr{H},$$
i.e., $\mu$ is both a Carleson {\em and} a reverse Carleson measure for $\mathscr{H}$. When $\mu$ is a discrete measure associated to a sequence of atoms in $\D$,  this sequence is often called a ``sampling sequence'' for $\mathscr{H}$ and there is a large literature on this subject \cite{Seip04}. Equivalent measures have also appeared in the context of ``dominating sets''. For example, it is often the case that $\mathscr{H}$ is naturally normed by an $L^2(\mu)$ norm, i.e., 
$$\|f\|_{\mathscr{H}} = \|f\|_{\mu} \quad \forall f \in \mathscr{H},$$ as is the case with the Hardy, Bergman, Paley-Wiener, Fock, and model spaces. For a Borel set $E$ contained in the support of $\mu$, one can ask whether or not the measure $\mu_{E} = \mu|_{E}$ satisfies 
\begin{equation}\label{784504rdsf}
\|f\|_{\mathscr{H}} \asymp \|f\|_{\mu_E} \quad \forall f \in \mathscr{H}.
\end{equation}
Such sets $E$ are called ``dominating sets" for $\mathscr{H}$. Historically, for the Bergman, Fock, and Paley-Wiener spaces, the first examples of reverse Carleson measures were obtained via dominating sets which, in these spaces, are naturally related with relative density, meaning that $E$ is never too far from the set on which the norm of the space is evaluated.

Though we will give a survey of reverse Carleson measures considered on a variety of Hilbert spaces, our main effort, and efforts of much recent work, will be on the sub-Hardy Hilbert spaces such as the model spaces and their de Branges-Rovnyak space generalizations. We will also comment on certain Banach space generalizations when appropriate, and in particular in connection with backward shift invariant subspaces. As it turns out the corresponding result from \cite{BFGHR} generalizes to $1<p<+\infty$. Indeed, this novel result follows from Baranov's proof as presented in \cite{BFGHR} and which we will reproduce in a separate appendix with the necessary modifications.

\section{The Hardy space} We assume the reader is familiar with the classical {\em Hardy space} $H^2$. For those needing a review, three excellent and well-known sources are \cite{Duren, Garnett, Koosis}. Functions in $H^2$ have radial boundary values almost everywhere on $\T$ and $H^2$ can be regarded as a closed subspace of $L^2$ via the ``vanishing negative Fourier coefficients" criterion. If $m$ is standard Lebesgue measure on $\T$, normalized so that $m(\T) = 1$, then $H^2$ is normed by the $L^2(m)$ norm $\|\cdot\|_{m}$. As expected, the subject of Carleson measures begins with this well-known theorem of Carleson \cite[Chap. I, Thm. 5.6]{Garnett}. 

\begin{Theorem}[Carleson]
For $\mu \in M_{+}(\D)$ the following are equivalent: 
\begin{enumerate}
\item[(i)] $\|f\|_{\mu} \lesssim \|f\|_{m}$ for all $f \in H^2$; 
\item[(ii)] $\|k_{\lambda}\|_{\mu} \lesssim \|k_{\lambda}\|_{m}$ for all $\lambda \in \D$, where $k_{\lambda}(z) = (1 - \overline{\lambda} z)^{-1}$ is the reproducing kernel for $H^2$; 
\item[(iii)] $\mu(S_{I}) \lesssim |I|$ for all arcs $I \subset \T$. 
\end{enumerate}
\end{Theorem}

This theorem can be generalized in a number of ways. First, the theorem works for the $H^p$ classes for $p \in (0, \infty)$ (with nearly the same proof). In particular, the set of Carleson measures for $H^p$ does not depend on $p$. Furthermore, notice that the original hypothesis of the theorem says that $\mu \in M_{+}(\D)$ and thus places no mass on $\T$. Since $H^2 \cap C(\D^{-})$ is dense in $H^2$ (finite linear combinations of reproducing kernels belong to this set), one can replace the condition $\|f\|_{\mu} \lesssim \|f\|_{m}$ for all $f \in H^2$ with the same inequality but with $H^2$ replaced with $H^2 \cap C(\D^{-})$. This enables an extension of Carleson's theorem to measures $\mu$ which could possibly place mass on $\T$ where the functions in $H^2$ are not initially defined. In the end however, this all sorts itself out since the Carleson window condition $\mu(S_{I}) \lesssim |I|$ implies that $\mu|_{\T} \ll m$ and so the integral in $\|f\|_{\mu}$ makes sense when one defines $H^2$ functions on $\T$ by their $m$-almost everywhere defined radial limits. Stating this all precisely, we obtain a revised Carleson theorem. 

\begin{Theorem}\label{Carleson-revised}
Suppose $\mu \in M_{+}(\D^{-})$. Then the following are equivalent: 
\begin{enumerate}
\item[(i)] $\|f\|_{\mu} \lesssim \|f\|_{m}$ for all $f \in H^2 \cap C(\D^{-})$; 
\item[(ii)] $\|k_{\lambda}\|_{\mu} \lesssim \|k_{\lambda}\|_{m}$ for all $\lambda \in \D$; 
\item[(iii)] $\mu(S_{I}) \lesssim |I|$ for all arcs $I \subset \T$. 
\end{enumerate}
Furthermore, when any of the above equivalent conditions hold, then $\mu|_{\T} \ll m$; the Radon-Nikodym derivative  $d \mu|_{\T}/dm$ is bounded; and $\|f\|_{\mu} \lesssim \|f\|_{m}$ for all $f \in H^2$. 
\end{Theorem}

We took some time to chase down this technical detail since, for other Hilbert spaces, we need to include the possibility that $\mu$ might place mass on the unit circle $\T$ and perhaps even have a non-trivial singular component (with respect to $m$). In fact, as we will see below when one discusses the works of Aleksandrov and Clark, there are Carleson measures, in fact isometric measures, for model spaces which are singular with respect to $m$. 

The reverse Carleson measure theorem for $H^2$ is the following \cite{HMNOC}. We include the proof since some of the ideas can be used to obtain a reverse Carleson measure for other sub-Hardy Hilbert spaces such as the model or de Branges-Rovnyak spaces (see Section \ref{Sec:Hbspaces}).

\begin{Theorem}\label{thmHMNO}
Let $\mu \in M_{+}(\D^{-})$.
Then the following assertions are equivalent:
\begin{enumerate}
\item[(i)] $\|f\|_{\mu} \gtrsim \|f\|_{m}$ for all $f \in H^2 \cap C(\D^{-})$; 
\item[(ii)] $\|k_{\lambda}\|_{\mu} \gtrsim \|k_{\lambda}\|_{m}$ for all $\lambda \in \D$; 
\item[(iii)] $\mu(S_{I}) \gtrsim |I|$ for every arc $I \subset \T$;
\item[(iv)] $\mbox{ess-inf  } d \mu|_{\T}/d m > 0$. 
\end{enumerate}
\end{Theorem}

\begin{proof}
(i) $\Rightarrow$ (ii) is clear.

(iii) $\Rightarrow$ (iv): Define
$$C = \inf_{I} \frac{\mu(S_{I})}{|I|}.$$ 
Let $I$ be an arc on $\T$ and take any (relatively) open set $O$ in $\D^{-}$  for which $I\subset O$. Then there exists an integer $N$ such that $h=|I|/N$ satisfies $S_{I,h}\subset O$ where $S_{I,h}$ is the modified Carleson window defined by 
\[
 S_{I,h}=\big\{z\in \D^{-} : 1-h\leqslant |z|\leqslant 1, \frac{z}{|z|} \in I\big\}\ .
\]
Divide $I$ into $N$ sub-arcs $I_k$ (suitable half-open except for the last one) such that $|I_k|=h$ (and hence
$S_{I_k,h}=S_{I_k}$). Then
\[
 \mu(S_{I,h})=\mu(\bigcup_{k=1}^N S_{I_k,h})=
\sum_{k=1}^N \mu(S_{I_k,h})\geqslant C \sum_{k=1}^N |I_k| = C |I|.
\]
For every (relatively) open set $O$
in $\D^{-}$  for which $I\subset O$ there exists $h>0$ such that
$S_{I,h}\subset O$. Since $\mu\in M_+(\D^-)$ is outer regular (see
\cite[Theorem 2.18]{Ru}) we have
\[
 \mu(I)=\inf\{\mu(O): I\subset O \text{ open in  } \D^{-}\} \geqslant \inf_{h>0}\mu(S_{I,h})
 \geqslant C |I|.
\]
We deduce that $m$ is absolutely
continuous with respect to $\mu|_{\T}$ and the
corresponding Radon-Nikodym derivative of $\mu$ is (essentially) bounded below by
$C$.

(iv) $\Rightarrow$ (i): Let 
$$A = \mbox{ess-inf  } d \mu|_{\T}/d m.$$ For all $f\in H^2\cap C(\D^{-})$,
\[
 \int_{\D^{-}}|f|^2 d\mu\geqslant \int_{\T}|f|^2 d\mu \geqslant A \int_{\T}|f|^2 dm.
\]

(ii) $\Rightarrow$ (iii): Let 
\begin{equation}\label{weiuhrwer}
K_{\lambda}(z) = \frac{k_{\lambda}(z)}{\|k_{\lambda}\|_{m}}
\end{equation} be the normalized reproducing kernel for $H^2$ and observe that since
$$\|k_{\lambda}\|_{m} = \frac{1}{\sqrt{1 - |\lambda|^2}},$$
the quantity 
$$|K_{\lambda}(z)|^2 = \frac{1 - |\lambda|^2}{|1- \overline{\lambda}z|^2}$$ is the Poisson kernel for the disk. Let 
$$B = \inf_{\lambda \in \D} \|K_{\lambda}\|_{\mu}^{2}$$ and note that $B > 0$ by hypothesis. 

Integrating over $S_{I,h}$ with respect to  area measure $dA$ on $\D$ we get
\begin{equation}\label{ineq}
 B  |I|\times h \leqslant \int_{S_{I,h}}\int_{\D^{-}} |K_{\lambda}|^2d\mu\; dA(\lambda)
  = \int_{\D^{-}} \int_{S_{I,h}}  \frac{1-|\lambda|^2}
{|1-\overline{\lambda}z|^2} dA(\lambda) d\mu(z).
\end{equation}

Set
\[
 \varphi_h(z)=\frac{1}{h}\int_{S_{I,h}}
 \frac{1-|\lambda|^2}{|1-\overline{\lambda}z|^2}dA(\lambda).
\]

We claim that
\[
 \lim_{h\to 0}\varphi_h(z)
 =
 \begin{cases} 
 1 &\mbox{if } z \in I^{\circ} \\ 
\tfrac{1}{2} & \mbox{if } z \in \partial I\\
0 & \text{if } z\in\D^{-}\setminus I^{-},
  \end{cases} 
\]
where $I^{-}$ denotes the closure, $I^{\circ}$ the interior, and $\partial I$ the boundary of the arc $I$.  Indeed, when $z\notin I^{-}$,  there are constants $\delta,h_0 >0$
such that for every $h \in (0, h_0)$ and for every $\lambda\in S_{I,h}$, we
have $|1-\overline{\lambda}z|\geqslant\delta>0$. The result now follows
from the estimate
\[
 0\leqslant \vp_h(z)=\frac{1}{h}\int_{S_{I,h}}
 \frac{1-|\lambda|^2}{|1-\overline{\lambda}z|^2}dA(\lambda)
 \leqslant \frac{1}{\delta^2}\frac{|I|\times h}{h}\times (2h)\lesssim h.
\]
When $z=e^{i\theta_0}\in I^{\circ}$,
then setting $\lambda=r e^{i\theta}$ for $\lambda\in S_{I,h}$
we have
\beqa
 \varphi_h(z)=\frac{1}{h}\int_{S_{I,h}}
 \frac{1-|\lambda|^2}{|1-\overline{\lambda}z|}dA(\lambda)
=\frac{1}{h}\int_{1-h}^1\int_I\frac{1-r^2}{|1-re^{-i\theta}z|^2}d\theta rdr.
\eeqa
Since $\mbox{dist}(z, \T \setminus I^{\circ}) > 0$ we see that when $r\to 1$ we have, via Poisson integrals, 
\[
\int_I\frac{1-r^2}{|1-re^{-i\theta}z|^2}d\theta
 =1-\int_{\T\setminus I}\frac{1-r^2}{|1-re^{-i\theta}z|^2}d\theta\to 1.
\]
Similarly, if can be shown that at the endpoints of $I$, $\varphi_{h}$ converges to $\tfrac{1}{2}$.
Hence $\varphi_h$ converges pointwise to a function comparable
to $\chi_{{I}}$, and $\varphi_h$ is uniformly bounded in
$h$. From \eqref{ineq} and the dominated convergence theorem we finally
deduce that
\[
 \mu({I})=\int_{\D^-} \chi_{{I}} d\mu \simeq \int_{\D^{-}}
 \lim_{h\to 0}\varphi_h(z)d\mu(z)=
 \lim_{h\to 0}\int_{\D^{-}}\varphi_h(z)d\mu(z)
 \gtrsim |I|\ . \qedhere
\]
\end{proof}

This theorem was proved in \cite{HMNOC} and extends to $1 < p < \infty$ with the same proof. There is a somewhat weaker version of this result in \cite{Queffelec}, appearing in the context of composition operators on $H^2$ with closed range, where the authors needed to assume from the onset that $\mu$ was a Carleson measure for $H^2$. Observe that in this theorem we
do not require absolute continuity
of the restriction $\mu|_{\T}$.
However, if we want to extend $\|f\|_{\mu} \gtrsim \|f\|_{m}$, originally assumed for $f \in H^2 \cap C(\D^{-})$, to all of $H^2$, then, in order for the integral in $\|f\|_{\mu}$ to make sense for every function in
$H^2$ (via radial boundary values), we need to impose the condition $\mu|_{\T} \ll m$.
Note that we are allowing the possibility that the integral $\|f\|_{\mu}$ be infinite for certain $f\in H^2$ when the Radon-Nikodym derivative
of $\mu|_{\T}$ is unbounded.

When $\mu \in M_{+}(\D^{-})$ one can combine Theorem \ref{Carleson-revised} and Theorem \ref{thmHMNO} to see that 
$$\|f\|_{\mu} \asymp \|f\|_{m} \; \; \forall f \in H^2 \iff \|k_{\lambda}\|_{\mu} \asymp \|k_{\lambda}\|_{m} \; \; \forall \lambda \in \D \iff \mu(S_{I}) \asymp |I| \; \; \forall I \subset \T.$$ One might ask what are the ``isometric measures'' for $H^2$, i.e., 
$\|f\|_{\mu} = \|f\|_{m}$ for all $f \in H^2$. Notice how this is a significantly stronger condition than $\|f\|_{m} \asymp \|f\|_{\mu}$. As it turns out, there is only one such isometric measure. 

\begin{Proposition}
Suppose $\mu \in M_{+}(\D^{-})$ and $\|f\|_{\mu} = \|f\|_{m}$ for all $f \in H^2 \cap C(\D^{-})$. Then $\mu = m$. 
\end{Proposition}

\begin{proof}
Indeed for each $n \in \N \cup \{0\}$ we have
$$1 = \|z^{n}\|^2_{m} = \int_{\D} |z|^{2 n} d \mu + \mu(\T).$$ 
Clearly, letting $n \to \infty$, we get $\mu(\T) = 1$. 
When $n = 0$ 
this yields $$\mu(\D) = 0 \quad \mbox{and} \quad \mu = \mu|\T.$$ 
By Carleson's criterion we see that $\mu \ll m$ and so $d \mu = h dm$, for some $h\in L^1(m)$. 
To conclude that $h$ is equal to one almost everywhere, apply the fact that $\mu$ is an isometric measure to the normalized reproducing kernels $K_{\lambda}$ (see \eqref{weiuhrwer})
 to see that 
$$1 = \int_{\T} \frac{1 - |\lambda|^2}{|1 - \overline{\zeta} \lambda|^2} h(\zeta) dm(\zeta) \quad \forall \lambda \in \D.$$ If we express the above as a Fourier series, we get
$$
1=\widehat h(0)+\sum_{n=1}^{\infty}\widehat h(-n)\overline{\lambda}^n+\sum_{n=1}^{\infty}\widehat h(n)\lambda^n,\qquad \lambda\in\D,
$$
and it follows that $h = 1$ $m$-a.e.~on $\T$. Thus $\mu = m$.
\end{proof}

\section{Bergman spaces}

The {\em Bergman space} $A^2$ is the space of analytic functions $f$ on $\D$ with finite norm
 \[
 \|f\|_{A^2}:= \left(\int_{\D}|f|^2 dA\right)^{\tfrac{1}{2}},
\]
where $dA=dxdy/\pi$ is normalized area Lebesgue measure on $\D$ \cite{MR2033762, HKZ}.
As with the Hardy space, we begin our discussion with the Carleson
measures for $A^2$. This was done by Hastings \cite{hastings}:

\begin{Theorem}
For $\mu \in M_{+}(\D)$ the following are equivalent:
\begin{enumerate}
\item[(i)] $\mu(S_{I}) \lesssim |I|^2$ for every arc $I\in\T$;
\item[(ii)] $\|f\|_{\mu} \lesssim \|f\|_{A^2}$ for every $f \in A^2$.
\end{enumerate}
\end{Theorem}

We also refer to \cite{HKZ} for further information about Carleson measures in Bergman spaces, including an equivalent restatement of this theorem involving pseudo-hyperbolic disks. In particular (see \cite[Theorem 2.15]{HKZ}) condition (i) is replaced by the condition: there exists an $r\in (0,1)$ such that 
\[
 \mu(D(a,r))\lesssim A(D(a,r)), \quad a \in \D, 
\] 
where
$$
D(a, r) = \left\{z \in \C: \left|\frac{z - a}{1 - \overline{z} a}\right| < r\right\}
$$
denotes a pseudo-hyperbolic disk of radius $r$ centered at $a$.
Observe that since $r$ is fixed,
we have $A(D(z,r))\asymp (1-|z|^2)^2$. Again, the geometric condition measures the 
amount of mass that $\mu$ places on a pseudohyperbolic
disk with respect to an intrinsic area measure of that disk. Hastings result was generalized by Oleinik and Pavlov, and Stegenga (see 
\cite{Luecking85} for the references).

Reverse Carleson embeddings for the Bergman spaces, and other closely related spaces, were discussed by Luecking \cite{Luecking81, Luecking85, Luecking88}. One of his first results in this direction concerns dominating sets, i.e., measures of the type $\chi_{G} dA$ (see \eqref{784504rdsf}). Here we have the following ``reverse'' of the inequality in Hasting's result (see \cite{Luecking81}). 

\begin{Theorem}\label{L-early}
Suppose $G$ is a (Lebesgue) measurable subset of $\D$. Then $\mu=\chi_{G} dA$ is a reverse Carleson measure for $A^{2}$ if and only if
$\mu(S_I) \gtrsim |I|^2$ for all arcs $I \subset \T$. 
\end{Theorem}

A similar result holds for the harmonic Bergman space \cite{LueHarm}. 
 We will discuss dominating sets again later when we cover model spaces (see Definition \ref{74yegsv0oijnc}). 

As it turns out, the general reverse Carleson measure result for Bergman spaces is more delicate \cite[Thm. 4.2]{Luecking85}. 
\begin{Theorem}\label{ThmLuecking85}
Let $\delta, \varepsilon>0$. 
Then there exists a $\beta>0$ with the
following property: Whenever $\mu \in M_{+}(\D)$ for which 
\begin{equation}\label{carlcondAp}
c = \sup_{a \in \D} \frac{\mu(D(a,1/2))}{A(D(a,1/2))} < \infty,
\end{equation}
and for which the set 
\begin{equation}\label{carlinvAp}
 G = \{z : \mu(D(z,\beta)) >  \varepsilon c A(D(z,\beta))\}
\end{equation}
satisfies 
\begin{equation}\label{reldens}
 m(G\cap S_I)\geqslant \delta |I|^2,
\end{equation} 
then $\|f\|_{A^2} \lesssim \|f\|_{\mu}$ for all $f \in A^2$.
\end{Theorem}

Notice how this theorem requires {\it a priori} that $\mu$ is
a Carleson measure for $A^2$ (via \eqref{carlcondAp}). The next two conditions tell us that the reverse Carleson condition \eqref{carlinvAp} must be
satisfied on a set which is, in a sense, relatively dense.
Moreover, the relative density condition in \eqref{reldens} should hold close
to the unit circle.


For simplicity we stated the results for the $A^2$ Bergman space. Analogous theorems (with the same proofs) are true for the $A^p$ Bergman spaces for $p \in (0, \infty)$. 

\section{Fock spaces} 

We briefly discuss Carleson and reverse Carleson measures for a space of entire functions - the Fock space. Here the conditions are a bit different since the functions are entire and there are no ``boundary conditions'' or ``Carleson boxes''.  

Let
$\phi$ be a subharmonic function on $\C$ (often called the weight) such that 
$$\frac{1}{c} \leqslant \Delta \phi\leqslant c$$ for some 
positive constant $c$. The {\em weighted Fock space} $\FF_{\phi}^2$ is the space of entire functions $f$ with finite norm
\[
 \|f\|_{\phi} =\left(\int_{\C}|f(z)|^2 e^{-2\phi(z)}dA(z)\right)^{\tfrac{1}{2}}.
\] 
Recall that $d A$ is Lebesgue area measure on $\C$. 
When $\phi(z) = |z|^2$, this space is often called the Bargmann-Fock space. A good primer for the Fock spaces is \cite{MR2934601}. There is also a suitable $L^p$ version of this space denoted by $\FF_{\phi}^{p}$ and the results below apply to these spaces as well.

The Carleson measures for $\FF^{2}_{\phi}$ were characterized by several authors (for various $\phi$) but the final, most general, result is found in Ortega-Cerd\`a \cite{OCsamp}. Below let 
$B(a, r) = \{z \in \C: |z - a| < r\}$ be the open ball in $\C$ centered at $a$ with radius $r$.  

\begin{Theorem}
For a locally finite positive Borel measure $\mu$ on $\C$, a weight $\phi$ as above, and 
$d\nu=e^{-2\varphi}d\mu$, the following are equivalent: 
\begin{enumerate}
\item[(i)] $\|f\|_{\nu} \lesssim \|f\|_{\phi}$ for all $f \in \FF_{\phi}^{2}$;
\item[(ii)] $\sup_{z \in \C} \mu(B(z, 1)) < \infty$.
\end{enumerate}
\end{Theorem}

The discussion of reverse Carleson measures for Fock spaces was begun by 
Janson-Peetre-Rochberg \cite{JPR}, again {\it via} dominating sets. 

\begin{Theorem}
For a weight $\phi$, a measurable set $E \subset \C$, and $d\nu=e^{-2\phi}\chi_E dA$, the following are equivalent: 
\begin{enumerate}
\item[(i)] $\|f\|_{\phi} \lesssim \|f\|_{\nu}$ for all $f \in \FF_{\phi}$;
\item[(ii)] there exists an $R>0$ such that $\inf_{z\in\C}A(E\cap B(z,R))>0$.
\end{enumerate}
\end{Theorem}

Condition (ii) is a relative density condition which, in a way, appeared in Theorem \ref{L-early}. We will meet such a condition again in Theorem \ref{hdshdshdhfdh} below when we discuss the Paley-Wiener space. 

 In \cite{OCsamp} Ortega-Cerd\`a
examined the measures $\mu$ on $\C$ for which
$$
 \|f\|_{\phi}^2 \asymp \int_{\C}|f(z)|^2 e^{-2 \phi(z)}d\mu(z)\qquad \forall f\in\FF_{\phi,2}.
$$
in other words, the ``equivalent measures'' 
for $\FF_{\phi}^2$. 
He called such measures {\em sampling measures}.
A special instance is when $$\mu=\sum_{n \geqslant 1} \delta_{\lambda_n},$$ where $
\Lambda= \{\lambda_n\}_{n \geqslant 1}$
is a sequence in the complex plane. In this case, $\{\lambda_n\}_{n\geqslant 1}$ is called a {\em sampling sequence}, meaning that 
\[
\|f\|_{\phi}^2\asymp \sum_{n\geqslant 1}|f(\lambda_n)|^2 e^{-2\phi(\lambda_n)} \qquad \forall f\in\FF_{\phi,2}.
\]
 
 Contrary to the approach in Bergman spaces, where Luecking characterized Carleson and reverse Carleson measures which, in turn, yielded information on sampling sequences, Ortega-Cerd\`a discretized $\mu$ to
 reduce the general case of
sampling measures to that of sampling {\em sequences}. These were characterized in a series
of papers by Seip, Seip-Wallst\'en, Berndtsson-Ortega-Cerd\`a and Ortega-Cerd\`a-Seip
(see \cite{Seip04} for these references). The main summary theorem is the following:

\begin{Theorem}\label{OC1}
A sequence $\Lambda \subset \C$ is a sampling sequence for $\FF_{\phi}^2$
if and only if the following two conditions are satisfied:
\begin{enumerate}
\item[(i)] $\Lambda$ is a finite union of uniformly separated sequences.
\item[(ii)] There is a uniformly separated subsequence $\Lambda' \subset \Lambda$ such that
\[
{\displaystyle  \varliminf_{r\to\infty} \inf_{z\in\C} \frac{\#(B(z,r)\cap\Lambda')}{\int_{B(z,r)}\Delta\phi d A}
  > \frac{1}{2\pi}.}
\]
\end{enumerate}
\end{Theorem}

To state the result in terms of sampling measures, we need to introduce some notation. For a large integer $N$ and positive numbers $\delta$ and $r$, decompose $\C$ into big squares $S$ of side-length $Nr$ and each square $S$ is itself decomposed into $N^2$ little squares of side-length $r$. Let $n(S)$ denote the number of little squares $s$ contained in $S$ such that $\mu(s)\geqslant\delta$. In terms of sampling measures, we have the following:

\begin{Theorem}\label{OC2}
The measure $\mu$ is a sampling measure if and only if the following conditions are satisfied:
\begin{enumerate}
\item[(i)] $\sup_{z \in \C} \mu(B(z, 1)) < \infty$;
\item[(ii)] There is an $r > 0$ and a grid consisting of squares of side-length $r$, an integer $N > 0$ 
and a positive number $\delta$ such that
\begin{equation}
\inf_S\frac{n(S)}{\int_S\Delta\phi\,dA}>\frac{1}{2\pi},
\end{equation}
where the infimum is taken over all squares $S$ consisting of $N^2$ little squares from 
the original grid.
\end{enumerate}
\end{Theorem}

Notice how (i) is a Carleson measure condition while (ii) is a reverse Carleson measure
condition.

To deduce Theorem \ref{OC1} from Theorem \ref{OC2}, Ortega-Cerd\`a first showed that
it is sufficient to consider the measure $\mu_1$ which is the part of $\mu$
supported only on the little squares $s$ for which $\mu(s)\geqslant \delta$ and then he discretized $\mu_1$ by $\mu_1^*=\sum_n \mu_1(s_n)\delta_{a_n}$, where $a_n$ is the center of
$s_n$. In order to show that $\mu_1$ is sampling exactly when $\mu_1^*$ is sampling, he used
a Bernstein-type inequality. This naturally links the problem of sampling measures
to the description of sampling sequences. Note that Bernstein inequalities also appear in the context of Carleson and reverse Carleson measures for model spaces (see Section~\ref{Sec:model-spaces}).

\section{Paley-Wiener space}

Though the Paley-Wiener space enters into the general discussion of model spaces presented in Section 6, we would like to present some older results which will help motivate the more recent ones. The {\em Paley-Wiener space} $PW$ is the space of entire functions $F$ of exponential type at most $\pi$, i.e., 
$$
\limsup_{|z|\to\infty}\frac{\log|F(z)|}{|z|}\leqslant\pi,
$$ 
and which are square integrable on $\R$. The norm on $PW$ is 
$$\|F\|_{PW} = \left(\int_{\R} |F(t)|^2 dt\right)^{\tfrac{1}{2}}.$$ A well-known theorem of Paley and Wiener \cite{deBranges68} says that $PW$ is the set of Fourier transforms of functions in $L^2$ 
which vanish on $\R \setminus [-\pi, \pi]$. 
Authors such as Kacnelson \cite{Kacnelson}, Panejah \cite{Panejah62, Panejah66},  and Logvinenko \cite{Logvinenko} examined Lebesgue measurable sets $E \subset \R$ for which 
$$\int_{\R} |F|^2 dt \asymp \int_{E} |F|^2 dt \quad \forall F \in PW.$$ 
Following \eqref{784504rdsf}, such sets will be called {\em dominating sets} for $PW$. Clearly we always have 
$$\int_{E} |F|^2 dt \leqslant \int_{\R} |F|^2 dt \quad \forall F \in PW.$$
The issue comes with the reverse lower bound. The summary theorem here is the following:

\begin{Theorem}\label{hdshdshdhfdh}
For a Lebesgue measurable set $E \subset \R$, the following are equivalent: 
\begin{enumerate}
\item[(i)] the set $E$ is a dominating set for $PW$;
\item[(ii)] there exists a $\delta>0$ and an $\eta>0$ such that
\begin{equation}\label{reldense}
 |E\cap [x-\eta,x+\eta]|\geqslant \delta, \quad \forall x\in\R.
\end{equation}
\end{enumerate}
\end{Theorem}

Notice how condition (ii) is a relative density condition we have met before when studying the Bergman and Fock spaces. 

Lin \cite{MR0187074} generalized the above result for measures $\mu$ on $\R$. We say that a positive locally finite measure $\mu$ on $\R$ is {\em $h$-equivalent to Lebesgue measure} if there exists a $K > 0$ such that 
$$\mu(x - h, x + h) \asymp h \quad \forall x\in\R, |x| > K.$$

\begin{Theorem}
Suppose $\mu$ is a locally finite Borel measure on $\R$. 
\begin{enumerate}
\item[(i)] There exists a constant $\gamma > 0$ such that if $\mu$ is $h$-equivalent to Lebesgue measure for some $h < \gamma$ then 
$$\int_{\R} |F|^2 dt \asymp \int_\R |F|^2 d \mu \quad \forall F \in PW.$$
\item[(ii)] If 
$$\int_{\R} |F|^2 dt \asymp \int_\R |F|^2 d \mu \quad \forall F \in PW,$$
then $\mu$ is $h$-equivalent to Lebesgue measure for some $h > 0$.
\end{enumerate}
\end{Theorem}

\section{Model spaces}\label{Sec:model-spaces}

A bounded analytic function $\Theta$ on $\D$ is called an {\em inner function} if the radial limits of $\Theta$ (which exist almost everywhere on $\T$ \cite{Duren}) are unimodular almost everywhere. Examples of inner functions include the Blaschke products $B_{\Lambda}$ with (Blaschke) zeros $\Lambda \subset \D$ and singular inner functions with associated (positive) singular measure $\nu$ on $\T$. In fact, every inner function is a product of these two basic types \cite{Duren}. 

Associated to each inner function $\Theta$ is a {\em model space}
$$\K_{\Theta} :=  (\Theta H^2)^\perp= \left\{f\in H^2: \int_{\T} f \overline{\Theta g} dm=0 \; \forall g\in H^2\right\}.$$
Model spaces are the generic (closed) invariant subspaces of $H^2$ for the backward shift operator
$$(S^{*} f)(z) = \frac{f(z) - f(0)}{z}.$$
 Moreover, the compression of the shift operator 
 $$(S f)(z) = z f(z)$$ to a model space is the so-called ``model operator'' for certain types of Hilbert space contractions. 

It turns out that the Paley-Wiener space $PW$ can be viewed as a certain type of model space. We follow \cite{MR0358396}. Let
$$\Psi(z) := \exp\left(2 \pi \frac{z + 1}{z - 1}\right)$$ be the atomic inner function with point mass at $z = 1$ and with weight $2 \pi$, 
$$(\mathscr{F} f)(x) := \frac{1}{\sqrt{2 \pi}} \int_{\R} e^{-i x t} f(t) dt,$$ the Fourier transform on $L^2(\R)$, and 
$$J: L^2(m) \to L^2(\R), \quad (J g)(x) = \frac{1}{\sqrt{\pi}} \frac{1}{x + i} f\big(\frac{x - i}{x + i}\big).$$
It is well known that  $\mathscr{F}$ is a unitary operator on $L^2(\R)$ and a change of variables will show that  $J$ is a unitary map from $L^2(m)$ onto $L^2(\R)$. It is also known \cite[p. 33]{MR0358396} that 
$$(\mathscr{F} J) \K_{\Psi} = L^2[0, 2 \pi].$$
If 
$$T: L^2[0, 2 \pi] \to L^2[-\pi, \pi], \quad (T h)(x) = h(x + \pi)$$ is the translation operator then
$$(T \mathscr{F} J) \K_{\Psi} = L^2[-\pi, \pi]$$ and 
$$(\mathscr{F} T \mathscr{F} J) \K_{\Psi} = PW.$$ Thus the Paley-Wiener space is an isometric copy of a certain model space in a prescribed way. 

 An important set associated with an inner function is its {\em boundary spectrum}
\begin{equation} \label{liminf-zero-set}
\sigma(\Theta):=\left\{ \xi\in\mathbb{T}:\;\varliminf_{z\to\xi}\left|\Theta\left(z\right)\right|=0\right\}.
\end{equation}

Using the factorization of $\Theta$ into a Blaschke product and a singular inner function, one can show that when $\sigma(\Theta) \not = \T$, there is a two-dimensional open neighborhood $\Omega$ containing $\T \setminus \sigma(\Theta)$ such that $\Theta$ has an analytic continuation to $\Omega$. 

Functions in model spaces can have more regularity than generic functions in $H^2$. Indeed, a result of Moeller  \cite{Mo} says every function in $\K_{\Theta}$ follows the behavior of its corresponding inner functions and has an analytic continuation to a two dimensional open neighborhood of  $\T \setminus \sigma(\Theta)$. In fact, one can say a little bit more. Indeed, for every $\xi \in \T \setminus \sigma(\Theta)$ the evaluation functional $E_{\xi} f = f(\xi)$ is continuous on $\K_{\Theta}$ with 
$$\|E_{\xi}\| = \sqrt{|\Theta'(\xi)|}.$$
Thus 
\begin{equation}\label{qwyw66xb111}
\sup_{\xi \in W} \|E_{\xi}\| < \infty
\end{equation} for any compact set $W \subset \D^{-} \setminus \sigma(\Theta)$.

 In terms of a measure $\mu \in M_{+}(\D^{-})$ being a Carleson measure for $\K_{\Theta}$, let us make the following simple observation.

\begin{Proposition}\label{t40proeihkqwj}
Suppose $\mu \in M_{+}(\D^{-})$ with support contained in $\D^{-} \setminus \sigma(\Theta)$. Then $\mu$ is a Carleson measure for $\K_{\Theta}$. 
\end{Proposition}

\begin{proof}
Let $W$ denote the support of $\mu$. From our previous discussion, every $f \in \K_{\Theta}$ has an analytic continuation to an open neighborhood of $W$. Furthermore, using \eqref{qwyw66xb111} we see that
$$\sup_{\xi \in W} |f(\xi)| \lesssim \|f\|_{m} \quad \forall f \in \K_{\Theta}.$$
It follows that $\|f\|_{\mu} \lesssim \|f\|_{m}$ and hence $\mu$ is a Carleson measure for $\K_{\Theta}$. 
\end{proof}

Two observations come from Proposition \ref{t40proeihkqwj}. The first is that there are Carleson measures for $\K_{\Theta}$ which are not Carleson for $H^2$ since $\mu(S_I) \lesssim |I|$ need not hold for all arcs $I \subset \T$. In fact one could even put point masses on $\T \setminus \sigma(\Theta)$. This is in contrast with the $H^2$ situation where we have already observed in Theorem \ref{Carleson-revised} that if $\mu\in M_{+}(\D^{-})$ is a Carleson measure for $H^2$, then $\mu|_{\T} \ll m$. The second observation is that if there is to be a Carleson testing condition like $\mu(S_{I}) \lesssim |I|$, the focus needs to be on the Carleson boxes $S_I$ which are, in a sense, close to $\sigma(\Theta)$.

 So far we have avoided the issue of making sense of the integrals $\|f\|_{\mu}$ for $f \in \K_{\Theta}$ when the measure $\mu$ could potentially place mass on $\T$. Indeed, we side stepped this in Proposition \ref{t40proeihkqwj}  by stipulating that the measure places no mass on $\sigma(\Theta)$, where the functions in $\K_{\Theta}$ are not well-defined. In order to consider a more general situation, and to adhere to the notation used in \cite{TV}, we make the following definition. 
 
\begin{Definition} 
A measure $\mu\in M_+(\D^{-})$ will be called {\em $\Theta$-admissible}  if the singular component of $\mu|_{\T}$ (relative to Lebesgue measure) is concentrated on $\T\setminus\sigma(\Theta)$. 
\end{Definition}

Since functions from $\K_\Theta$ are continuous (even analytic) on this set, it follows that for $\Theta$-admissible measures and functions $f\in \K_\Theta$, the integral $\|f\|_{\mu}$ makes sense.

 
As was done with the Hardy spaces in Theorem \ref{Carleson-revised}, one could state the definition of a Carleson measure for $\K_{\Theta}$ to be a $\mu \in M_{+}(\D^{-})$ for which 
\begin{equation}\label{A-emb}
\|f\|_{\mu} \lesssim \|f\|_{m} \quad \forall f \in \K_{\Theta} \cap C(\D^{-}).
\end{equation}
Indeed, an amazing result of Aleksandrov \cite{Alek} says that $\K_{\Theta} \cap C(\D^{-})$ is dense in $\K_{\Theta}$ and so this set makes a good ``test set'' for the Carleson (reverse Carleson) condition. Furthermore, if $\mu \in M_{+}(\D^{-})$ and \eqref{A-emb} holds, then $\mu$ is $\Theta$-admissible, every function in $\K_{\Theta}$ has radial limits $\mu|_{\T}$-almost everywhere on $\T$, and $\|f\|_{\mu} \lesssim \|f\|_{m}$ for every $f \in \K_{\Theta}$. 

Carleson measures for $\K_{\Theta}$ were discussed in the papers of Cohn \cite{Cohn} and Treil and Volberg \cite{TV}. Their theorem is stated in terms of 
\begin{equation} \label{SLS}
\Omega(\Theta, \varepsilon) := \{z \in \D: |\Theta(z)| < \varepsilon\},\qquad 0<\varepsilon<1,
\end{equation}
the {\em sub-level sets} for $\Theta$. Note that boundary spectrum $\sigma(\Theta)$ is contained in the closure of any $\Omega(\Theta,\varepsilon)$, $0<\varepsilon<1$. 

\begin{Theorem}\label{7775w882727}
Suppose $\mu \in M_{+}(\D^{-})$ and define the following conditions:
\begin{enumerate}
\item[(i)]  $\mu(S_{I}) \lesssim |I|$ for all arcs $I \subset \T$ for which $S_I \cap\Omega(\Theta,\varepsilon) \not = \varnothing$;
\item[(ii)] $\mu$ is a Carleson measure for $\K_{\Theta}$;  
\item[(iii)] $\mu$ is $\Theta$-admissible and $\|k_\lambda^\Theta\|_\mu\lesssim \|k_\lambda^\Theta\|_m$ holds for every $\lambda\in\mathbb D$.
\end{enumerate}
Then $(i)\Longrightarrow (ii)\Longrightarrow (iii)$. Moreover, if for some $\varepsilon\in (0,1)$, the sub-level set 
$\Omega(\Theta,\varepsilon)$ is connected, then $(i)\Longleftrightarrow (ii)\Longleftrightarrow (iii)$. 
%
\end{Theorem}

The condition that $\Omega(\Theta, \varepsilon)$ is connected for some $\varepsilon  \in (0, 1)$ is often called the {\em connected level set condition} (CLS). Cohn \cite{Cohn} proved that if $\Omega(\Theta, \varepsilon)$ is connected and $\delta \in (\varepsilon, 1)$, then $\Omega(\Theta, \delta)$ is also connected. Any finite Blaschke product, the atomic inner function 
$$\Theta(z) = \exp\left(\frac{z + 1}{z - 1}\right),$$ and the infinite Blaschke product whose zeros are $\{1 - r^n\}_{n \geqslant 1}$, where $0 < r < 1$, satisfy this connected level set condition.

The sufficient condition appearing in assertion $(i)$ of Theorem~\ref{7775w882727} is, in general, not necessary. More precisely, Treil and Volberg \cite{TV} proved that this condition is necessary for the embedding of $\K_\Theta$ into $L^2(\mu)$ if and only if $\Theta\in (CLS)$. Nazarov--Volberg \cite{Nazarov-Volberg} proved that the RKT (reproducing kernel thesis) for Carleson embeddings for $\K_\Theta$ is, in general, not true. In \cite{Baranov-JFA05}, Baranov obtained a significant extension of the Cohn and Volberg--Treil results, introducing a new point of view based on certain Bernstein-type inequalities. Quite recently, in answering a question posed by Sarason~\cite{MR2363975}, Baranov--Besonnov--Kapustin \cite{BKK} clarified a nice link between Carleson measures for $\K_\Theta$ and an interesting class of operators -- the truncated Toeplitz operators -- which have received much attention in the last few years \cite{MR2363975}.

We turn to reverse Carleson measures.
Since the main reverse embedding result for model spaces, or backward shift invariant subspaces, is new in the non Hilbert situation we will state this theorem  for $1<p<+\infty$. 
In this more general situation we need the following definition 
\[
 \K_{\Theta}^p=H^p\cap \Theta \overline{H^p_0},
\]
where $H^p_0=zH^p$ is the space of functions in $H^p$ vanishing at $0$. The above intersection is to be understood on the circle. We will denote $L^p(\mu)=L^p(\overline{\D},\mu)$. 

The reverse embedding theorem goes along the lines of Treil-Volberg for which we need the following additional notation: 
given an arc $I\subset \T$ and a number $n>0$, we define
the amplified arc  $nI$ as the arc with the same center as $I$ but 
with length $n \times m(I)$.

\begin{Theorem}\label{thm:first-reversed-embedding}
Let $\Theta$ be inner, $\mu \in M_{+}(\D^{-})$, and $\varepsilon \in (0, 1)$. There exists an $N = N(\Theta, \varepsilon) > 1$ such that 
if
\begin{equation}\label{eq:reversed-condition-first}
 \mu(S_{I}) \gtrsim m(I)
\end{equation}
for all arcs
$I \subset \T$ satisfying  
$$S_{N I} \cap \Omega(\Theta,\varepsilon)\neq\varnothing,$$
then 
\begin{equation}\label{eq:reversed-equation}
\|f\|_{L^p(m)} \lesssim \|f\|_{L^p(\mu)} \qquad \forall f \in \K_{\Theta}^p\cap C(\D^{-}).
\end{equation}
\end{Theorem}

This theorem is a more general version than the one appearing in \cite[Theorem 2.1]{BFGHR}, not only in that it works for $p\neq 2$, but also it does not require the (direct) Carleson condition (which is not really needed in the proof). 
It was initially proved in \cite{BFGHR} for (CLS)-inner function using a 
perturbation argument from \cite[Corollary 1.3 and the proof of Theorem 1.1]{Baranov-II}, but Baranov
provided a proof (found in \cite{BFGHR}) based on Bernstein inequalities and which does not require the CLS condition. 
As it turns out, Baranov's proof does not use specific Hilbert space tools and generalizes to the situation
$1<p<+\infty$. The proof of this theorem is reproduced in the appendix. Apart from the natural changes to switch from $p=2$ to general $p$, we also include explicitely an argument from \cite{Queffelec} which was not detailed in the original proof in \cite{BFGHR} in order to show here
that the direct Carleson measure condition is not required.

\begin{Corollary}
Under the hypotheses of Theorem \ref{thm:first-reversed-embedding}, and if, moreover, the
measure $\mu$ is assumed to be $\Theta$-admissible, then \eqref{eq:reversed-equation} extends to
all of $\K_{\Theta}^p$.
\end{Corollary}

 Our second reverse Carleson result involves the notion of a dominating set for $\K_{\Theta}$, defined in \eqref{784504rdsf} and discussed earlier for the Bergman and Fock spaces.

\begin{Definition}\label{74yegsv0oijnc}
A (Lebesgue) measurable subset $\Sigma\subset\mathbb{T}$, with $m\left(\Sigma\right)<1$,
is called a \emph{dominating set} for $\K_{\Theta}$ if 
$$
\int_{\T} |f|^2dm \lesssim \int_{\Sigma}|f|^2 dm \quad \forall f \in \K_{\Theta}.
$$
\end{Definition}

This is equivalent to saying that the measure $d \mu  = \chi_{\Sigma} dm$ is a reverse Carleson measure for $\K_{\Theta}$. Here we list some observations concerning dominating sets for model spaces.  We will use the
following notation for sets $A$, $B$ and a point $x$:
$$
d(A, B) := \inf\{|a - b|: a \in A, b \in B\}, \quad d(x, A) :=d(\{x\},A).
$$

Throughout the list below we will assume that $\Theta$ is inner and $\sigma(\Theta)$ is its
boundary spectrum from \eqref{liminf-zero-set}. 
All of these results can be found in \cite[Section 5]{BFGHR}.

\begin{enumerate}
\item[(i)] If $\Sigma$ is a dominating set for $\K_{\Theta}$ then, 
for every $\zeta\in\sigma(\Theta)$, we have
$d(\zeta,\Sigma)=0$.  
\item[(ii)] If $\Sigma$ is a dominating set for $\K_{\Theta}$ then
$d(\Sigma,\sigma(\Theta))=0$.
\item[(iii)] Let $\zeta\in\sigma(\Theta)$ and $\Sigma$ dominating. Then there exists an
$\alpha>0$ such that for every sequence $\lambda_n\to \zeta$ with 
$\Theta(\lambda_n)
\to 0$, there is an integer $N$ with
\[
 m(\Sigma\cap I_{\lambda_n}^{\alpha})\gtrsim m(I_{\lambda_n}^{\alpha}),
 \quad n\geqslant N.
\]
In the above, $I_{\lambda}^{\alpha}$ is the subarc of $\T$ centered at $\frac{\lambda}{|\lambda|}$ with length $\alpha(1-|\lambda|)$. 
\item[(iv)]  Every open subset  $\Sigma$  of $\T$ such that $\sigma(\Theta)\subset\Sigma$ and $m(\Sigma)<1$ is a dominating set for $\K_{\Theta}$.

\item[(v)] Let $\Theta$ be an inner function such that $m(\sigma(\Theta))=0$. Then for every $\varepsilon \in (0, 1)$ there is a dominating set $\Sigma$ for $\K_{\Theta}$ such that $m(\Sigma)<\varepsilon$. In particular, this is true for (CLS)-inner functions.
\item[(vi)] If $\sigma(\Theta)=\T$ and  if $\Sigma$ is a dominating set for $\K_{\Theta}$ then $\Sigma$ is dense in $\T$.  
\item[(vii)] There exists a Blaschke product $B$ with $\sigma(B)=\T$
and an open subset
$\Sigma\subsetneq\mathbb{T}$ 
dominating 
for $\K_{B}$. 

\item[(vi)] Every model space admits a dominating set.
\end{enumerate}

Theorem \ref{thm:first-reversed-embedding} shows, in the special case of the Paley-Wiener space, that when \eqref{reldense} is satisfied  for sufficiently small $\eta$, then $E$ is a dominating set for $PW$. 


For reverse Carleson measures there is the following result from \cite{BFGHR}. 

\begin{Theorem}\label{thm:second-reversed-embedding}
Let $\Theta$ be an inner function, $\Sigma$ be a dominating set for
$\K_{\Theta}$, and $\mu \in M_{+}(\D^{-})$. 
Suppose that
\[
\inf_{I}\frac{\mu(S_{I})}{m\left(I\right)}>0,
\]
where the above infimum is taken over all arcs $I \subset \T$ such that $I\cap\Sigma\neq\varnothing$.
Then
\begin{equation}\label{reversed-inequality}
\|f\|_m \lesssim \|f\|_{\mu} \qquad \forall f \in \K_{\Theta}\cap C(\D^{-}).
\end{equation}
\end{Theorem}

\begin{Corollary}
Under the hypotheses of Theorem \ref{thm:second-reversed-embedding}, and if moreover the
measure $\mu$ is assumed to be $\Theta$-admissible, then the inequality in \eqref{reversed-inequality} extends to all of $\K_{\Theta}$.
\end{Corollary}

For the Hardy space, the reverse Carleson measures were characterized by the reverse reproducing kernel thesis, i.e., 
$\|k_{\lambda}\|_{m} \lesssim \|k_{\lambda}\|_{\mu}$ for all $\lambda \in \D$. For model spaces, however,  the reverse reproducing kernel thesis is a spectacular failure \cite{HMNOC}. 

\begin{Theorem}
Let $\Theta$ be an inner function that is not a finite Blaschke product.
Then there exists a measure $\mu \in M_{+}(\T)$ such that $\mu$ is a Carleson measure for $\K_{\Theta}$, the reverse estimate on reproducing kernels
$k^\Theta_\lambda$,
$$
\label{1}
\|k^\Theta_\lambda\|_{\mu} \gtrsim \|k^\Theta_\lambda\|_m
\qquad \forall \lambda\in \mathbb{D},
$$
is satisfied, 
but $\mu$ is not a reverse Carleson measure for $\K_{\Theta}$.
\end{Theorem}

Let us see this counterexample worked out in the special case of the Paley-Wiener space $PW$, which, recall from our earlier discussion, is isometrically isomorphic to the model space $\K_{\Theta}$
with 
$$\Theta(z) = \exp\big(2 \pi \frac{z + 1}{z - 1}\big).$$
Consider the sequence $S=\{x_n\}_{n\in\Z\setminus\{0\}}$, where
\[
 x_n=
 \begin{cases}
  n+1/8\ &\textrm{if $n$ is even}\\
  n-1/8\ &\textrm{if $n$ is odd.}
 \end{cases}
\]
By the Kadets-Ingham theorem \cite[Theorem D4.1.2]{Nik}, $S$
is a minimal sampling (or complete interpolating) sequence if we include the point $0$. Since $S$ is not
sampling, the
discrete measure 
$$\mu:=\sum_{n\neq 0}\delta_{x_n}$$
does not satisfy the reverse inequality 
$$\|f\|_{L^2(\R)}\lesssim \|f\|_{L^2(\mu)} \quad \forall f \in PW.$$

However, the $L^2(\mu)$-norm of the normalized
reproducing kernels
\[
 K_{\lambda}(z)=c_\lambda
\sinc(\pi(z-\lambda))=c_\lambda\frac{\sin(\pi(z-\lambda))}{
\pi(z-\lambda)},\qquad c_\lambda^2\simeq (1+ |\Im\lambda|)e^{-2\pi |\Im \lambda|},
\]
is uniformly bounded from below. Indeed, if $\lambda$ is such that $|\Im
\lambda|>1$ then
$$|\sin(\pi(x_n-\lambda))|\simeq e^{\pi|\Im \lambda|},$$ and hence
\[
\int_{\C} |K_\lambda(x)|^2 d\mu(x)
 =\sum_{n\neq 0} c_{\lambda}^2\left|\frac{\sin(\pi(x_n-\lambda))}{\pi(x_n-\lambda)}
 \right|^2\simeq
\sum_{n\neq 0} \frac
{|\Im\lambda|}{|x_n-\lambda|^2}\simeq 1.
\]
Thus it is enough to consider points $\lambda\in\C$ with $|\Im\,
\lambda|\leqslant 1$. Let $x_{n_0}$ be the point of $S$ closest to
$\lambda$. Then there is $\delta>0$, independent of $\lambda$, such
that
\[
 \int_{\C} |K_\lambda(x)|^2 d\mu(x)
 =\sum_{n\neq 0} |K_\lambda (x_n)|^2
 \geqslant \left|\frac{\sin(\pi(x_{n_0}-\lambda))}{\pi(x_{n_0}-\lambda)}\right|^2\geqslant \delta.
\]
It is interesting to point out that $\mu$ is a Carleson measure for $PW$ since $S$ is in a strip and separated.

As was asked for the Paley-Wiener space $PW$, what are the $\mu \in M_{+}(\T)$ for which 
$$\|f\|_{m} \asymp \|f\|_{\mu} \quad \forall f \in \K_{\Theta}?$$ In \cite{Volberg-81} Volberg generalized the previous results and gave a complete answer for general model spaces and absolutely continuous measures $d \mu = w dm$, where $w \in L^{\infty}(\mathbb{T})$, $w \geqslant 0$. Let
\[
 \widehat w(z)=\int_{\T}w(\zeta)\frac{1-|z|^2}{|z-\zeta|^2}\,dm(\zeta),\qquad z\in\D,
\]
be the Poisson integral of $w$ and note that $\widehat w$ is harmonic (and positive) on $\D$ and has radial boundary values equal to $w$ $m$-almost everywhere \cite{Duren}. 

\begin{Theorem}\label{thm:Volberg} Let $d\mu=wdm$, with $w\in L^\infty(\T)$, $w \geqslant  0$, and let $\Theta$ be an inner function. Then the following assertions are equivalent:
\begin{itemize}
\item[(i)]  $\|f\|_{m} \asymp \|f\|_{\mu}$ for all $f \in \K_{\Theta}$;
\item[(ii)] if $\{\lambda_n\}_{n\geqslant 1}\subset\D$, then 
\[
\lim_{n\to \infty}\widehat w(\lambda_n)=0 \implies \lim_{n\to \infty}|\Theta(\lambda_n)|=1;
\]
\item[(iii)]  
$
\inf \{\widehat w(\lambda)+|\Theta(\lambda)|: \lambda \in \D\}>0.
$
\end{itemize}
\end{Theorem}

In particular, this theorem applies to the special case when $d\mu=\chi_\Sigma dm$, with $\Sigma$ a Borel subset of $\T$.
However the conditions obtained from Volberg's theorem are not expressed directly in terms of a density condition as was the case for $PW$ (see Theorem~\ref{hdshdshdhfdh}). It is natural to ask if we can obtain a characterization of dominating sets for $\K_\Theta$ in terms of a relative density. Dyakonov answered this question in \cite{MR1111913}. In the following result, $\mathscr{H}^2$ is the Hardy space of the upper-half plane $\{\Im z > 0\}$, $\Psi$ is an inner function on $\{\Im z > 0\}$, and $\mathscr{K}_{\Psi} = (\Psi \mathscr{H}^2)^{\perp}$ is a model space for the upper-half plane.

\begin{Theorem}
For an inner function $\Psi$ on $\{\Im z > 0\}$ the following are equivalent:
\begin{enumerate}
\item[(i)] $\Psi' \in L^{\infty}(\R)$;
\item[(ii)] Every Lebesgue measurable set $E \subset \R$ for which these exists an $\delta>0$ and an $\eta>0$ such that
$$
 |E\cap [x-\eta,x+\eta]|\geqslant \delta \quad \forall x\in\R
$$
is dominating for the model space $\mathscr{K}_{\Psi}$.
\end{enumerate}
\end{Theorem}
In the case corresponding to the Paley-Wiener space $PW$, $\Psi(z)=e^{2i\pi z}$ and thus $|\Psi'(x)|=2\pi$ on $\R$. As was shown by Garnett \cite{Garnett}, the condition $\Psi'\in L^{\infty}(\R)$ is equivalent to one of the following two conditions: 
\begin{enumerate}
\item[(i)] $\exists h>0$ such that 
$$
\inf\{|\Psi(z)|:0<\Im(z)<h\}>0;
$$
\item[(ii)] $\Psi$ is invertible in the Douglas algebra $[H^\infty,e^{-ix}]$ (the algebra generated by $H^\infty$ and the space of bounded uniformly continuous functions on $\R$).
\end{enumerate}
For instance, the above conditions are satisfied when $\Psi(z)=e^{iaz}B(z)$, where $a>0$ and $B$ is an interpolating Blaschke product satisfying $\dist(B^{-1}(\{0\}),\R)>0$ (e.g., the zeros of $B$ are $\{n + i\}_{n \in \Z}$). 

What happens if we were to replace the condition 
$$\|f\|_{m} \asymp \|f\|_{\mu} \quad \forall f \in \K_{\Theta}$$ with the stronger condition 
$$\|f\|_{m} = \|f\|_{\mu} \quad \forall f \in \K_{\Theta}.$$
Such ``isometric measures'' 
were characterized by Aleksandrov \cite{Al98} (see also \cite{BFGHR}).

\begin{Theorem}\label{ThmAleks}
For $\mu\in M_+(\T)$  the following assertions are equivalent:
\begin{itemize}
\item[(i)] $\|f\|_{\mu} = \|f\|_{m}$ for all $f \in \K_{\Theta}$;
\item[(ii)] $\Theta$ has non-tangential boundary values $\mu$-almost
everywhere on $\T$ and
\[
 \int_{\T}\left|\frac{1-\overline{\Theta(z)}\Theta(\zeta)}{1-\overline{z}
 \zeta}\right|^2d\mu(\zeta)=\frac{1-|\Theta(z)|^2}{1-|z|^2},
 \quad z\in\D;
\]
\item[(iii)] there exists a $\varphi\in H^{\infty}$ such that $\|\varphi\|_{\infty}\leqslant
1$ and
\begin{eqnarray}\label{ThmAleksClark}
 \int_{\T}\frac{1-|z|^2}{|\zeta-z|^2}d\mu(\zeta)
 =\Re\left(\frac{1+\varphi(z)\Theta(z)}{1-\varphi(z)\Theta(z)}\right),
 \quad z\in\D.
\end{eqnarray}
\end{itemize}
\end{Theorem}

The condition in \eqref{ThmAleksClark} says that  $\mu$ is one of the so-called {\em Aleksandrov-Clark measures} for $b=\phi \Theta$. It is known that the operator $V_b:L^2(\mu)\lra \HH(b)=K_{\Theta}\oplus \Theta\HH(\phi)$ introduced in \eqref{OpVb} below is an onto partial isometry, which is isometric on $H^2(\mu)$, the closure of the polynomials in $L^2(\mu)$ (see Section \ref{Sec:Hbspaces} for more on $\HH(b)$-spaces and Aleksandrov-Clark measures). By a result of Poltoratski \cite{Poltoratskii}, $V_bg=g$ $\mu_S$-a.e. where $\mu_S$ is the singular part of $\mu$ with respect to $m$. In particular, when $\phi$ is inner, then $\HH(b)=\K_{\Theta\phi}=\K_{\Theta}\oplus\Theta \K_{\phi}$ and $\mu=\mu_S$ is singular, and hence for every $f=V_bg\in \K_{\Theta}$, where $g\in H^2(\mu)$, we have
\[
 \|f\|_m=\|V_bg\|_m=\|g\|_{\mu}= \|f\|_{\mu}.
\] 
When $\phi$ is not inner, Aleksandrov proves Theorem \ref{ThmAleks} by using  the above fact for inner functions along with the fact that the isometric measures form a closed subset of the Borel measures $M(\T)$ in the topology $\sigma(M(\T),C(\T))$.


L. de Branges \cite{deBranges68} proved a version of Theorem \ref{ThmAleks} for meromorphic inner functions and Krein \cite{Krein-Gorbachuck} obtained a characterization of  isometric measures for $\K_{\Theta}$ using more operator theoretic langage. 


\section{de Branges-Rovnyak spaces}\label{Sec:Hbspaces}

These spaces are generalizations of the model spaces. Let 
$$H^{\infty}_{1} = \{f \in H^{\infty}: \|f\|_{\infty} \leqslant 1\}$$ be the {\em closed unit ball in $H^{\infty}$}. Recall that when $\Theta$ is inner, the model space $\K_{\Theta}$ is a closed subspace of $H^2$ with reproducing kernel function 
$$k^{\Theta}_{\lambda}(z) = \frac{1 - \overline{\Theta(\lambda)} \Theta(z)}{1 - \overline{\lambda} z},\qquad \lambda,z\in\D.$$ Using this as a 
guide, one can, for a given $b \in H^{\infty}_{1}$, define the {\em de Branges-Rovnyak space} $\HH(b)$ to be the unique reproducing kernel Hilbert space of analytic functions on $\D$ for which 
$$k^{b}_{\lambda}(z) = \frac{1 - \overline{b(\lambda)} b(z)}{1 - \overline{\lambda} z},\qquad \lambda,z\in\D,$$ is the reproducing kernel \cite{Paulsen}. Note that the function $K(z,\lambda):=k_\lambda^b(z)$ is positive semi-definite on $\D$, i.e.,
$$
\sum_{i,j=1}^n \overline{a_i}a_j K(\lambda_i,\lambda_j)\geqslant 0,
$$
for all finite sets $\{\lambda_1,\dots,\lambda_n\}$ of points in $\D$ and all complex numbers $a_1,\dots,a_n$. Hence, we can associate to it a reproducing kernel Hilbert space and the above definition makes sense. There is an equivalent definition of $\HH(b)$ via defects of certain Toeplitz operators \cite{Sarason}.

 It is well known that though these spaces play an important role in understanding contraction operators, the norms on these $\HH(b)$ spaces, along with the elements contained in these spaces, remain mysterious. When $\|b\|_{\infty} < 1$ (i.e., $b$ belongs to the interior of $H^{\infty}_{1}$), then $\HH(b) = H^2$ with an equivalent norm. When $b$ is an inner function, then $\HH(b) = \K_{b}$ with the $H^2$ norm. For general $b \in H_{1}^{\infty}$, $\HH(b)$ is contractively  contained in $H^2$ and this space is often called a ``sub-Hardy Hilbert space'' \cite{Sarason}. The analysis of these $\HH(b)$ spaces naturally splits into two distinct cases corresponding as to whether or not $b$ is an extreme function for $H^{\infty}_{1}$, equivalently, $\log (1 - |b|) \not \in L^1(m)$.

When $b \in H^{\infty}_{1}$ is non-extreme, there is a unique outer function $a \in H^{\infty}_{1}$ such that $a(0) > 0$ and 
\begin{equation}\label{PMate}
|a(\xi)|^2 + |b(\xi)|^2 = 1 \quad \mbox{$m$-a.e. $\xi \in \T$}.
\end{equation}
Such $a$ is often called the {\em Pythagorean mate} for $b$ and the pair $(a, b)$ is called a {\em Pythagorean pair}.

There is the, now familiar,  issue of boundary behavior of $\HH(b)$ functions when defining the integrals $\|f\|_{\mu}$ in the Carleson and reverse Carleson testing conditions. With the model spaces (and with $H^2$) there is a dense set of continuous functions  for which one can sample in order to test the Carleson  ($\|f\|_{\mu} \lesssim \|f\|_{m}$) and reverse Carleson conditions ($\|f\|_{m} \lesssim \|f\|_{\mu}$). For a general $\HH(b)$ space however, it is not quite clear whether or not $\HH(b)\cap C(\D^{-})$ is even non-zero. In certain circumstances, for example when $b$ is non-extreme or when $b$ is an inner function, $\HH(b)\cap C(\D^{-})$ is actually dense in $\HH(b)$. For general extreme $b$, this remains unknown. Thus we are forced to make some definitions. 

\begin{Definition}\label{defn:admissible}
For $\mu \in M_{+}(\D^{-})$ we say that an analytic function $f$ on $\D$ is $\mu$-\emph{admissible} if the non-tangential limits of $f$ exist $\mu$-almost everywhere on $\T$. We let $\mathscr{H}(b)_{\mu}$ denote the set of $\mu$-admissible functions in $\mathscr{H}(b)$. 
\end{Definition}

With this definition in mind, if $f \in \mathscr{H}(b)_{\mu}$, then defining $f$ on the carrier of $\mu|_{\T}$ via its non-tangential boundary values, we see that $\|f\|_{\mu}$ is well defined with a value in $[0, +\infty]$.

Of course when $\mu$ is carried on $\D$, i.e., $\mu(\T) = 0$, then $\HH(b)_\mu=\HH(b)$. Hence Definition \ref{defn:admissible} only comes into play when $\mu$ has part of the unit circle $\T$ in its carrier. Note that  $\mathscr{H}(b) = \mathscr{H}(b)_{m}$ since $\mathscr{H}(b)\subset H^2$.
However, there are often other $\mu$, even ones with non-trivial singular parts on $\T$ with respect to $m$, for which $\mathscr{H}(b) = \mathscr{H}(b)_{\mu}$.  The Clark measures associated with an inner function $b$ 
have this property \cite{BFGHR, CRM}.  

\begin{Definition}
 A measure $\mu \in M_{+}(\D^{-})$ is a \emph{Carleson measure} for $\mathscr{H}(b)$ if $\HH(b)_{\mu} = \HH(b)$ and $\|f\|_{\mu} \lesssim \|f\|_{b}$ for all $f \in \mathscr{H}(b)$.
 \end{Definition}
 
 When $b \equiv 0$, i.e., when $\HH(b) = H^2$ then, as a consequence of Carleson's theorem (see Theorem \ref{Carleson-revised}) for $H^2$, we see that when $\mu$ satisfies $\mu(S_{I}) \lesssim |I|$ for all arcs $I$, then $\mu|_{\T} \ll m$ and so $\HH(b)_{\mu} = \HH(b)$. When $b$ is an inner function, recall a discussion following \eqref{A-emb} which says that if the Carleson testing condition $\|f\|_{\mu} \lesssim \|f\|_{m}$ holds for all $f \in \HH(b)\cap C(\D^{-})$, then $\HH(b)_{\mu} = \HH(b)$. So in these two particular cases, the delicate issue of defining the integrals in $\|f\|_{\mu}$ for $f \in \HH(b)$ seems to sort itself out. For general $b$, we do not have this luxury. 
  
Lacey et al.~\cite{LW} solved the longstanding problem of characterizing the two-weight inequalities for 
Cauchy transforms.  Let us take a moment to
indicate how their results can be used to discuss Carleson measures for $\mathscr{H}(b)$. 
Let $\sigma$ be the Aleksandrov-Clark measure associated with $b$, that is the unique $\sigma \in M_{+}(\T)$ satisfying
\[
\frac{1-|b(z)|^2}{|1-b(z)|^2}=\int_\T \frac{1-|z|^2}{|z-\zeta|^2}\,d\sigma(\zeta),\quad z\in\D.
\]
Let $V_b:L^2(\sigma)\longrightarrow \HH(b)$ be the operator defined by 
\begin{equation}\label{OpVb}
(V_b f)(z)=(1-b(z))\int_\T \frac{f(\zeta)}{1-\bar\zeta z}\,d\sigma(\zeta)=(1-b(z))(C_\sigma f)(z),
\end{equation}
where 
$C_{\sigma}$ is the Cauchy transform
$$(C_{\sigma} f)(z) = \int_{\T} \frac{f(\zeta)}{1 - \overline{\zeta} z} d\sigma(\zeta).$$
It is known  \cite{Sarason} that $V_b$ is a partial isometry from $L^2(\sigma)$ onto $\HH(b)$ and 
$$\ker V_b = \ker C_{\sigma} = (H^2(\sigma))^\perp.$$ Here $H^2(\sigma)$ denotes the closure of polynomials in $L^2(\sigma)$ and the $\perp$ is in $L^2(\sigma)$. 
As a consequence, since every function $f\in \HH(b)$ can be written as
$f=V_bg$ for some $g\in H^2(\sigma)$, $\mu$
is a Carleson measure for $\HH(b)$ if and only if 
\[
\|V_b g\|_{\mu}=\|f\|_{\mu}\lesssim \|f\|_b=\|V_b g\|_b=\|g\|_\sigma \quad \forall g\in H^2(\sigma).
\]
Setting $\nu_{b,\mu}:=|1-b|^2\mu$, we have $$\|V_b g\|_{\mu}^2=
\int_{\D^-}|1-b|^2 |C_\sigma g|^2\,d\mu=\|C_\sigma g\|_{\nu_{b,\mu}}^2.$$
This 
yields the following:

\begin{Theorem}\label{LW-C}
Let $\mu \in M_{+}(\D^-)$, $b$ a $\mu$-admissible
function in $H^\infty_1$, and $\nu_{b,\mu}:=|1-b|^2\mu.$ Then the following are equivalent: 
\begin{enumerate}
\item[(i)] $\mu$ is a Carleson measure for $\HH(b)$;
\item [(ii)] The Cauchy transform $C_\sigma$  is a bounded operator from $L^2(\sigma)$ into $L^2(\D^-,\nu_{b,\mu})$, where $\sigma$ is the Aleksandrov-Clark measure associated with $b$.
\end{enumerate}
\end{Theorem}
We refer the reader to \cite[Theorem 1.7]{LW} for a description of the boundedness of the Cauchy transform operator $C_{\sigma}$. However, it should be noted that the  characterization of Carleson measures for $\HH(b)$, obtained combining Theorem~\ref{LW-C} and \cite[Theorem 1.7]{LW}, is not purely geometric. 

The following result from \cite{BFM}, similar in flavor to Theorem \ref{7775w882727}, discusses the Carleson measures for $\HH(b)$.

\begin{Theorem}
For $b \in H^{\infty}_{1}$ and $\varepsilon\in (0,1)$ define 
$$
\Omega(b,\varepsilon):=\{z\in\D:|b(z)|<\varepsilon\},
$$
$$\Sigma(b):=\left\{\zeta\in\T:\varliminf_{z\to\zeta}|b(z)|<1\right\},
$$
$$\widetilde\Omega(b,\varepsilon) :=\Omega(b,\varepsilon)\cup\Sigma(b).$$
Let $\mu \in M_{+}(\D^{-})$ and define the following conditions:
\begin{enumerate}
\item[(i)]  $\mu(S_{I}) \lesssim |I|$ for all arcs $I \subset \T$ for which $I \cap \widetilde\Omega(b,\varepsilon) \not = \varnothing$;
\item[(ii)] $\HH(b)_{\mu} = \HH(b)$ and $\|f\|_{\mu} \lesssim \|f\|_{b}$ for all $f \in \HH(b)$;
\item[(iii)] $\HH(b)_{\mu} = \HH(b)$ and $\|k_\lambda^b\|_\mu\lesssim \|k_\lambda^b\|_b$ for all $\lambda\in\mathbb D$.
\end{enumerate}
Then $(i)\Longrightarrow (ii)\Longrightarrow (iii)$. Moreover, suppose there exists an $\varepsilon\in (0,1)$ such that $\Omega(b,\varepsilon)$ is connected and its closure contains $\Sigma(b)$. Then $(i)\Longleftrightarrow (ii)\Longleftrightarrow (iii)$. 
%
%
%
%
%
\end{Theorem} 

It should be noted here that, contrary to the inner case, the containment $\Sigma(b)\subset\clos(\Omega,\varepsilon)$  is not, in general, automatic. Indeed, when $b(z)=(1+z)/2$, one can easily check that the above containment is not satisfied. 

Here is a complete description of the Carleson measures for a very specific $b$ \cite{RCDR}. Note that if $b$ is a non-extreme rational function (e.g., rational but not a Blaschke product), one can show that the Pythagorean mate $a$ from \eqref{PMate} is also a rational function.  
 
 \begin{Theorem}\label{thm1.11}
Let $b \in H^{\infty}_{1}$ be rational and non-extreme and let $\mu \in M_{+}(\D^{-})$. Then the following assertions are equivalent: 
\begin{enumerate}
\item  $\mu$ is a Carleson measure for $\HH(b)$;
\item  $|a|^2\,d\mu$ is a Carleson measure for $H^2$. 
\end{enumerate}
\end{Theorem}

If $b(z)=(1+z)/2$ then $a(z)=(1-z)/2$ and, if $\mu$ is the measure supported on $(0,1)$ defined by $d\mu(t)=(1-t)^{-\beta}dt$, for $\beta\in (0,1]$, we can use Theorem~\ref{thm1.11} to see that $\mu$ is Carleson measure for $\HH(b)$. However, $\mu$ is not a Carleson measure for $H^2$. One can see this by considering the arcs $I_\vartheta=(e^{-i\vartheta},e^{i\vartheta})$, $\vartheta\in (0,\pi/2)$, and observing that 
$$
\sup_\vartheta \frac{\mu(S(I_\vartheta))}{|I_\vartheta|}=\infty.
$$

If $b$ is a $\mu$-admissible function, then so are all of the reproducing kernels $k^{b}_{\lambda}$ (along with finite linear combinations of them) and thus, with this admissibility assumption on $b$, $\mathscr{H}(b)_{\mu}$ is a dense linear manifold in $\mathscr{H}(b)$. This motivates our definition of a reverse Carleson measure for $\HH(b)$.

\begin{Definition}
For $\mu \in M_{+}(\D^{-})$ and $b \in H^{\infty}_{1}$, we say that $\mu$ is a \emph{reverse Carleson measure for} $\mathscr{H}(b)$ if $\HH(b)_{\mu}$ is dense in $\HH(b)$ 
and $\|f\|_{b} \lesssim \|f\|_{\mu}$ for all $f \in \mathscr{H}(b)_{\mu}$. In this definition, we allow the possibility for the integral $\|f\|_{\mu}$ to be infinite. 
\end{Definition}

Here is a reverse Carleson measure result from \cite{RCDR} which focuses on the case when $b$ is non-extreme. 
 
\begin{Theorem}\label{MainThmIntro}
Let $\mu \in M_{+}(\D^{-})$ and let $ b \in H^{\infty}_{1}$ be non-extreme and $\mu$-admissible. If $h=d\mu|_{\T}/dm$, then the following assertions are equivalent: 
\begin{enumerate}
\item[(i)] $\mu$ is a reverse Carleson mesure for $\mathscr{H}(b)$;
\item[(ii)]  $\|k^{b}_\lambda\|_b \lesssim \|k^{b}_{\lambda}\|_{\mu}$ for all $\lambda\in\D$;
\item[(iii)] $d\nu:=(1 - |b|)d\mu$ 
satisfies
\begin{equation*}
\inf_{I}\frac{\nu\left(S_I\right)}{m(I)}>0;
\end{equation*}
\item[(iv)] $\mathrm{ess}\inf_{\T} (1 - |b|) h>0$. 
\end{enumerate}
\end{Theorem}
The proof of this results is in the same spirit as Theorem~\ref{thmHMNO}. Also note that the condition $(iv)$ implies that $(1-|b|)^{-1}\in L^1$. As a consequence of this observation, we see that if $b \in H_1^\infty$ is non-extreme and such that $(1-|b|)^{-1} \not \in L^1$, then there are {\em no} reverse Carleson measures for $\HH(b)$.

As was done with many of the other spaces discussed in this survey, one can say something about the equivalent measures for $\HH(b)$ \cite{RCDR}. 

\begin{Theorem}\label{ireughjfef}
Let $b\in H^\infty_1$ be  non-extreme and $\mu\in M_+(\D^{-})$. Then the following are equivalent:
\begin{enumerate}
\item[(i)] $\HH(b)_{\mu} = \HH(b)$ and $\|f\|_\mu \asymp \|f\|_b$ for all $f\in\HH(b)$;
\item[(ii)] The following conditions hold: 
\begin{enumerate}
\item $a$ is $\mu$-admissible, 
\item $(a,b)$ is a corona pair, i.e., 
$$\inf\{|a(z)| + |b(z)|: z \in \D\} > 0;$$
\item  $|a|^2$ satisfies the Muckenhoupt $(A_2)$ condition, i.e.,
$$
\sup_{I} \left(\frac{1}{m(I)}\int_{I}|a|^{-2}\,dm\right)\left(\frac{1}{m(I)}\int_{I}|a|^2\,dm\right)<\infty,
$$
where $I$ runs over all  subarcs of $\T$;
\item $d\nu :=|a|^2\,d\mu$ satisfies
$$0 < \inf_{I}\frac{\nu\left(S_{I}\right)}{m(I)} \leqslant \sup_{I}\frac{\nu\left(S_I\right)}{m(I)} < \infty,$$
where the infimum and supremum above are taken over all open arcs $I$ of $\T$.
\end{enumerate}
\end{enumerate}
\end{Theorem}
One should note that if $(a,b)$ is a corona pair and $|a|^2\in (A_2)$, then $\HH(b)=\MM(a)$, where $\MM(a)=aH^2$ equipped with the range norm, i.e., $\|ag\|_{\MM(a)}=\|g\|_m$, for any $g\in H^2$  \cite[IX-5]{Sa}. Hence the above result says that it is possible to obtain an equivalent norm on $\HH(b)$ expressed in terms of an integral only when $\HH(b) = \MM(a)$. 

Surely an example is important here: Let $a(z) :=  c_{\alpha} (1 - z)^{\alpha}$, where $\alpha \in (0, 1/2)$ and $c_{\alpha}$ is suitable chosen so that $a \in H^{\infty}_{1}$. 
When $0 < \alpha < 1/2$, one can show that $|a|^2$ satisfies the $(A_2)$ condition. Choose $b$ to be the outer function in $H^{\infty}_{1}$ satisfying $|a|^2 + |b|^2 = 1$ on $\T$. Standard theory \cite{MR0289784}, using the fact that $a$ is H\"{o}lder continuous on $\D^{-}$, will show that $b$ is continuous on $\D^{-}$. From here it follows that $(a, b)$ is a corona pair. If $\sigma \in M_{+}(\D^{-})$ is any Carleson measure for $H^2$, then one can show that $d \mu := |a|^{-2} dm + d \sigma$ satisfies the conditions of Theorem \ref{ireughjfef}.

For $\HH(b)$ spaces when $b$ non-extreme, the isometric measures: $\|f\|_{\mu} = \|f\|_{b}$ for all $f \in \HH(b)$, are not worth discussing as illustrated by the following result. 

\begin{Theorem}
\label{no-isometry}
When $b$ is non-constant and non-extreme, there are no positive isometric measures for $\HH(b)$.
\end{Theorem}

Also not worth discussing for general $\HH(b)$ spaces is the notion of dominating sets \cite{RCDR}: $E \subset \T$, $0 < m(E) < 1$, for which
$$\|f\|_{b}^{2} \lesssim \int_{E} |f|^2 dm \quad \forall f \in \HH(b).$$ Indeed, we have the following:

\begin{Theorem}
Let $b \in H^{\infty}_{1}$ such that $\HH(b)$ has a dominating set. Then either $b$ is an inner function or $\|b\|_{\infty} < 1$.
\end{Theorem}

As one can see, the case for extreme $b$ seems to be very much open. When $b$ is inner, much has been said about the Carleson and reverse Carleson measures for $\HH(b) = \K_b$. When $b$ is extreme but not inner, there are a few things one can say \cite{RCDR} but there is much work to be done to complete the picture. 

\section{Harmonically weighted Dirichlet spaces}

For $\mu \in M_{+}(\T)$ let 
$$\varphi_{\mu}(z) = \int_{\T} \frac{1 - |z|^2}{|\xi - z|^2} d\mu(\xi), \quad z \in \D,$$
denote the Poisson integral of $\mu$. The {\em harmonically weighted Dirichlet space} $\mathscr{D}(\mu)$  \cite{KEFR, Ri} is the set of all analytic functions $f$ on $\D$  for which 
$$\int_{\D} |f'|^2 \varphi_{\mu} dA < \infty,$$
where $d A  = dx dy/\pi$ is normalized planar measure on $\D$. Notice that when $\mu = m$, we have $\varphi_{\mu} \equiv 1$ and $\mathscr{D}(\mu)$ becomes the classical Dirichlet space \cite{KEFR}. One can show that $\mathscr{D}(\mu) \subset H^2$ \cite[Lemma 3.1]{Ri} and the norm $\|\cdot\|_{\DD(\mu)}$ given by  
$$\|f\|_{\mathscr{D}(\mu)}^2 := \int_{\T} |f|^2 dm + \int_{\D} |f'|^2 \varphi_{\mu} dA$$ makes $\mathscr{D}(\mu)$ into a reproducing kernel Hilbert space of analytic functions on $\D$. It is known that both the polynomials as well as the linear span of the Cauchy kernels form dense subsets of $\mathscr{D}(\mu)$ \cite[Corollary 3.8]{Ri}. 

When $\zeta \in \T$ and $d\mu=\delta_\zeta$, a result from  \cite{SaLoc} shows that 
\[
 \mathscr{D}(\delta_{\zeta})=\HH(b),
\]
where $w_0=(3-\sqrt{5})/2$ and 
\begin{equation}\label{DefFctb}
 b(z)=\frac{(1-w_0)\overline{\zeta}z}{1-w_0\overline{\zeta}z}.
\end{equation}
Furthermore, the norms on these spaces are the same. In fact, these are the only harmonically weighted Dirichlet spaces which are equal to an $\HH(b)$ space with equal norm \cite{CGR}. In \cite{CostRans} it was shown that if 
\begin{equation}\label{wefroiuer}
\mu = \sum_{j = 1}^{n} c_j \delta_{\zeta_j}, \quad c_j > 0, \zeta_j \in \T
\end{equation} is a finite linear combination of point masses on $\T$ and $a$ is the unique polynomial with $a(0) > 0$ and with  simple zeros at $\zeta_j$ (and no other zeros) and $b$ is the Pythagorean mate for $a$ (which must also be a polynomial), then $\HH(b) = \DD(\mu)$ with equivalent norms. In this case we can use Theorem \ref{thm1.11} to obtain a characterization of the Carleson measures for $\DD(\mu)$:

 \begin{Theorem}
For $\mu$ as in \eqref{wefroiuer} and $\nu \in M_{+}(\D^{-})$, the following assertions are equivalent: 
\begin{enumerate}
\item[(i)]  $\nu$ is a Carleson measure for $\DD(\mu)$;
\item[(ii)]  $\prod_{i=1}^n\left|z-\zeta_i \right|^2 d\nu$ is a Carleson measure for $H^2$. 
\end{enumerate}
\end{Theorem}

This  result appeared in \cite{MR3097551} (see also \cite{chacon}). In fact, Theorem 6.1 from \cite{MR3097551} shows that the above conditions are equivalent to 
$$\|k^{\DD(\mu)}_{\lambda}\|_{\nu} \lesssim \|k^{\DD(\mu)}_{\lambda}\|_{\DD(\mu)} \quad \forall \lambda \in \D.$$
In other words, at least when $\mu$ is a linear combination of point masses, the reproducing kernel thesis characterizes the Carleson measures for $\DD(\mu)$. 

%
%
The discussion of reverse Carleson measures for $\DD(\mu)$ is dramatically simpler since they do not exist! Indeed, suppose that $\nu\in M_+(\D^{-})$ and $\|f\|_\mu\lesssim \|f\|_\nu$ for all $f\in\DD(\mu)$. In particular, this is true for the monomials $z^n$, $n\geqslant 0$. But $\|z^n\|_\nu\lesssim 1$ and $\|z^n\|_\mu^2=1+n\mu(\T)$, which gives a contradiction when $n$ tends to $\infty$.

We point out some related results from \cite{chacon} which discuss a type of reverse Carleson measure for $\DD(\mu)$ spaces except that the definitions of ``reverse Carleson measures'' and ``sets of domination'' (dominating sets) are quite different, and not equivalent, to ours. 

\section{Appendix}

We reproduce here an adaption to $1<p<+\infty$ 
of Baranov's proof  as presented in \cite[Section 7]{BFGHR} and which is based on the Bernstein-type 
inequalities in model spaces he obtained in \cite{Baranov-JFA05,Baranov-09}.
It uses a Whitney type decomposition of $\T\setminus \sigma(\Theta)$. 
Let $\varepsilon>0$, let $\delta\in (0,1/2)$ and let 
\[
d_\varepsilon(\zeta)=d(\zeta,\Omega(\Theta,\varepsilon)),
\]
where we recall that $\Omega(\Theta,\varepsilon)=\{z\in\D:|\Theta(z)|<\varepsilon\}$. Since 
\[
\int_{\T\setminus\sigma(\Theta)}d_\varepsilon^{-1}(\zeta)\,dm(\zeta)=\infty,
\]
we can choose a sequence of arcs $I_k$ with pairwise disjoint interiors such that $\bigcup_k I_k=\T\setminus\sigma(\Theta)$ and 
\[
\int_{I_k}d_\varepsilon^{-1}(\zeta)\,dm(\zeta)=\delta.
\]
In this case\footnote{Note that such a system of arcs was also considered in \cite{Baranov-09} for $\delta=1/2$.}
\begin{equation}\label{eq:distance-level-set}
\frac{1-\delta}{\delta}m(I_k)\leq d(I_k,\Omega(\Theta,\varepsilon))\leq\frac{1}{\delta}m(I_k).
\end{equation}
Indeed, by the definition of $I_k$, there exists $\zeta_k\in I_k$ such that $d_\varepsilon(\zeta_k)=\frac{1}{\delta}m(I_k)$, whence for any $\zeta\in I_k$, we have 
\[
d_\varepsilon(\zeta)\geq d_\varepsilon(\zeta_k)-m(I_k)\geq \frac{1-\delta}{\delta}m(I_k).
\]
It follows from \eqref{eq:distance-level-set} that 
\[
m(I_k)^{1/(p-1)}\int_{I_k}d_\varepsilon^{-q}(u)\,dm(u)\leq \left(\frac{\delta}{1-\delta}\right)^q.
\]
Now recall the definition of the weight involved in
the Bernstein-type inequality 
\[
w_r(z)=\|(k_z^\Theta)^2\|_s^{-\frac{r}{r+1}},
\]
where $1\leq r<\infty$ and $s$ is the conjugate exponent of $r$. 
(We point out a misprint in the definition of $w_r$ in \cite{BFGHR} where the square was omitted
inside the norm.) Later on we will choose 
$r$ such that $1\leq r<p$. Then it is shown in \cite[Lemmas 4.5 \& 4.9]{Baranov-JFA05} that 
\[
w_r(\zeta)\geq C_0 d_\varepsilon(\zeta),
\]
where $C_0$ depends only on $r$ and $\varepsilon$ (but not on $\Theta$). Thus
\begin{equation}\label{eq:weight-integral}
m(I_k)^{1/(p-1)}\int_{I_k} w_r^{-q}(\zeta)\,dm(\zeta)\leq  C\delta^q. 
\end{equation}

Let $I_k^{(j)}$, $j=1, \dots 4$ be the quarters of $I_k$
and let $S_k^{(j)}$ be the parts of $S_k$ lying over  $I_k^{(j)}$. Thus, $S_k = \bigcup_{j=1}^4 S_k^{(j)}$
(note that $S_k^{(j)}$ are not standard Carleson windows). By \eqref{eq:distance-level-set}, we have 
\[
S(NI_k^{(j)})\cap \Omega(\Theta,\varepsilon)\neq \emptyset
\]
as soon as $N>\frac{8}{\delta}$. 
This will be the choice of $N$ in the Theorem.
Suppose now that
\[
 A:=\inf_I\frac{\mu(S(I))}{m(I)}>0,
\]
where the infimum is taken over all arcs $I\subset\T$ with
$S(NI)\cap \Omega(\Theta,\varepsilon)\neq\emptyset$.
 Then we have 
\[
\mu(S_k^{(j)})\geq \mu(S(I_k^{(j)}))\geq A m(I_k^{(j)}).
\]

Now let $f\in \K_\Theta$ be continuous in $\D\cup\T$. By the mean value property, there exists $s_k^{(j)}\in S_k^{(j)}$ such that 
\begin{equation}\label{eq:int-S-k-j}
\int_{S_k^{(j)}} |f|^pd\mu = |f(s_k^{(j)})|^p \mu(S_k^{(j)})  \geq A m(I_k^{(j)}) \cdot |f(s_k^{(j)})|^p.
\end{equation}
Denote by 
\[
\mathfrak{J}_k^{i,j}=\int_{I_k^{(i)}}|f(u)-f(s_k^{(j)})|^p\,dm(u).
\]
Then we have
\begin{eqnarray*}
\lefteqn{\sum_k \int_{I_k}|f|^p\,dm =\sum_k \left(\int_{I_k^{(1)}}|f(u)|^p+\int_{I_k^{(2)}}|f(u)|^p+\int_{I_k^{(3)}}|f(u)|^p+\int_{I_k^{(4)}}|f(u)|^p \right)\,dm(u)} \\
 &&\leq c_p\sum_k (\mathfrak{J}_k^{1,3}+\mathfrak{J}_k^{2,4}+\mathfrak{J}_k^{3,1}+\mathfrak{J}_k^{4,2})\\
&&+c_p\sum_k\left(|f(s_k^{(3)})|^p m(I_k^{(1)})+|f(s_k^{(4)})|^p m(I_k^{(2)})+|f(s_k^{(1)})|^p m(I_k^{(3)})+|f(s_k^{(2)})|^p m(I_k^{(4)}) \right).
\end{eqnarray*}
Since $m(I_k^{(1)})=m(I_k^{(2)})=m(I_k^{(3)})=m(I_k^{(4)})$, we get with \eqref{eq:int-S-k-j}
\[
\sum_k \int_{I_k}|f|^p\,dm\leq c_p\sum_k (\mathfrak{J}_k^{1,3}+\mathfrak{J}_k^{2,4}+\mathfrak{J}_k^{3,1}+\mathfrak{J}_k^{4,2})+c_pA^{-1}\|f\|_{L^p(\mu)}^p.
\]

Let us now estimate $\sum_k \mathfrak{J}_k^{1,3}$. We have 
\[
\mathfrak{J}_k^{1,3}=\int_{I_k^{(1)}}|f(u)-f(s_k^{(3)})|^p\,dm(u)=\int_{I_k^{(1)}}\left|\int_{[s_k^{(3)},u]}f'(v)\,|dv|\right|^p\,dm(u),
\]
where $[s_k^{(3)},u]$ denotes the interval with endpoints $s_k^{(3)}$ and $u$ and $|dv|$ stands for the Lebesgue measure on this interval. Using H\"older's inequality, we obtain 
\[
\mathfrak{J}_k^{1,3}\leq \int_{I_k^{(1)}}\left(\int_{[s_k^{(3)},u]}|f'(v)|^pw^p_r(v)\,|dv|\right) \left( \int_{[s_k^{(3)},u]}w_r^{-q}(v)\,|dv|  \right)^{p/q}\,dm(u).
\] 
Now recall that the norms of reproducing kernels in model spaces have a certain monotonicity along the radii. More precisely, let $q>1$. Then it is shown in \cite[Corollary 4.7.]{Baranov-JFA05} that there exists $C=C(q)$ such that for any $z=\rho e^{it}$ and $\tilde z=\tilde\rho e^{it}$ with $0\leq \tilde\rho\leq \rho$, we have
\begin{equation}\label{norm-rep-kern-monotone}
\|k_{\tilde z}^\Theta\|_\alpha\leq C(q) \|k_z^\Theta\|_\alpha
\end{equation}
(which we use here for $\alpha=2q$).
Using \eqref{norm-rep-kern-monotone}, \eqref{eq:weight-integral} and the fact that the angle\footnote{That explains why we choose a decomposition with $\mathfrak{J}_k^{i,j}$, $i\neq j$, since in this case the interval $[s_k^{(j)},u]$, $u\in I_k^{(i)}$, will never be orthogonal to the boundary.} between $[s_k^{(3)},u]$ and $\T$ is separated from $\frac{\pi}{2}$, we conclude that 
\[
  \left( \int_{[s_k^{(3)},u]}w_r^{-q}(v)\,|dv|  \right)^{p/q}
 \le C\delta^p(m(I_k)^{-1/(p-1)})^{p/q}=C\frac{\delta^p}{m(I_k)},
\]
Hence
\[
\sum_k\mathfrak{J}_k^{1,3}\leq C 
\delta^p \sum_k \frac{1}{m(I_k)}\int_{I_k^{(1)}} \int_{[s_k^{(3)},u]}|f'(v)|^p w_r^p(v)\,|dv|\,dm(u).
\]
Again just by the mean value property, there exists $u_k\in I_k^{(1)}$ such that 
\[
\sum_k \frac{1}{m(I_k)}\int_{I_k^{(1)}}\int_{[s_k^{(3)},u]}|f'(v)|^p w_r^p(v)\,|dv|\,dm(u)=\frac{1}{4}\sum_k\int_{[s_k^{(3)},u_k]}|f'(v)|^p w_r^p(v)\,|dv|.
\]
Now note that the measure $\sum_k m_{[s_k^{(3)},u_k]}$ (sum of Lebesgue measures on the intervals) 
is a Carleson measure with a uniform bound on the Carleson constant independent of the location of $u_k\in I_k^{(1)}$ and $s_k^{(3)}\in S_k^{(3)}$ (and of $\delta$). Then by the Bernstein inequality
\cite[Theorem 1.1]{Baranov-JFA05}, we have 
\[
\sum_k\int_{[s_k^{(3)},u_k]}|f'(v)|^p w_r^p(v)\,|dv|\leq C\|f\|_p^p,
\] 
which gives 
\[
\sum_k \mathfrak{J}_k^{1,3}\leq C \delta^p \|f\|_p^p.
\]
Using similar estimates for the other terms $\sum_k \mathfrak{J}_k^{2,4}$, $\sum_k \mathfrak{J}_k^{3,1}$ and $\sum_k \mathfrak{J}_k^{4,2}$, we obtain 
\[
\sum_k \int_{I_k}|f|^p\,dm\leq C \delta^p \|f\|_p^p+c_pA^{-1}\|f\|_\mu^p.
\]
Finally we consider the integral over $\sigma(\Theta)=\T\setminus\bigcup_k I_k$. For this,
as indicated in \cite{BFGHR}, we use an argument from \cite{Queffelec} which we would like
to make more explicit here, thereby showing that the direct Carleson measure condition is indeed not required in the argument. Recall that $f\in \K_{\Theta}^p\cap C(\T)$. Also, it is clear that we can assume
$\|f\|_{L^p(\mu)}\neq 0$. 
By uniform
continuity there exists $\rho>0$ such that for every $z,z'\in \overline{\D}$ with $|z-z'|<\sqrt{2}\rho$
we have
\[
 |f(z)-f(z')|\le\frac{\|f\|_{L^p(\mu) }}{ A^{1/p}}.
\]
Now there exists a sequence of arcs $(J_k)$ (not necessarily open) with pairwise disjoint interiors
such that $m(J_k)<\rho$, and with $\sigma(\Theta)\subset \bigcup_k J_k$ and
$S(J_k)\cap \Omega(\Theta,\varepsilon)\neq \emptyset$. Let $z_k\in S(J_k)$ such that $|f(z_k)|$ is
the smallest value of $|f|$ in $S(J_k)$ and let $\zeta_k\in J_k$  be such that $|f(\zeta_k)|$ is 
the biggest value of $|f|$ on $J_k$. Observe also that the diameter of $S(J_k)$ is less than
$\sqrt{2}\rho$. 
Then
\begin{eqnarray*}
 \|f\|_{L^p(\mu)}&\ge& \left[\sum_{k}\int_{S(J_k)}|f|^pd\mu\right]^{1/p}
 \ge \left[\sum_{k}|f(z_k)|^p\mu(S(J_k))\right]^{1/p}\\
 &\ge&  \left[A\sum_{k}|f(z_k)|^pm(J_k)\right]^{1/p}\\
 &\ge&A^{1/p}
   \left[\sum_{k}|f(\zeta_k)|^pm(J_k)\right]^{1/p}-
  \left[A\sum_{k}|f(z_k)-f(\zeta_k)|^pm(J_k)\right]^{1/p}\\
 &\ge&A^{1/p}
   \left[\int_{\bigcup_{k} J_k}|f|^pdm\right]^{1/p}-\left[A\frac{\|f\|^p_{L^p(\mu)}}
   { A}m(\bigcup_{k}J_k)\right]^{1/p}\\
 &\ge& A^{1/p}\left[\int_{\sigma(\Theta)}|f|^p dm\right]^{1/p}-\|f\|_{L^p(\mu)}.
\end{eqnarray*}
As a result, setting $\tilde{A}= 2^p/A$,
\begin{equation}\label{eq:integral-over-spectrum}
\int_{\sigma(\Theta)}|f|^p\,dm\leq \tilde A \|f\|_{L^p(\mu)},
\end{equation}
Thus we finally obtain
\[
\|f\|_p^p \leq \frac{2^p+c_p}{A}\|f\|_{L^p(\mu)}^p+C\delta^2 \|f\|_p^p, 
\]
that is 
\[
(1-C\delta^p)\|f\|_p^p\leq  \frac{2^p+c_p}{A}\|f\|_{L^p(\mu)}^p. 
\]
It remains to choose $\delta$ small enough.

\bibliographystyle{plain}

\bibliography{referencesRC}

\end{document}